\documentclass{AIMS}

\usepackage{amsmath}
  \usepackage{paralist}
  \usepackage{graphics} %% add this and next lines if pictures should be in esp format
  \usepackage{epsfig} %For pictures: screened artwork should be set up with an 85 or 100 line screen
\usepackage{graphicx}  \usepackage{epstopdf}%This is to transfer .eps figure to .pdf figure; please compile your paper using PDFLeTex or PDFTeXify.
 \usepackage[colorlinks=true]{hyperref}
\usepackage{cleveref}
\usepackage{enumitem}
\hypersetup{urlcolor=blue, citecolor=red}

\newtheorem{theorem}{Theorem}[section]
\newtheorem{corollary}[theorem]{Corollary}

\newtheorem{lemma}[theorem]{Lemma}
\newtheorem{proposition}[theorem]{Proposition}

\theoremstyle{definition}
\newtheorem{definition}[theorem]{Definition}
\newtheorem{remark}[theorem]{Remark}

\newtheorem*{brc}{Boundary Regularity Condition}

\newcommand{\abs}[1]{\left\vert#1\right\vert} % Absolute value
\newcommand{\norm}[1]{\left\|#1\right\|} % Norm
\newcommand{\pare}[1]{\left(#1\right)} % Parenthesis
\newcommand{\braces}[1]{\left\{#1\right\}} % Braces
\newcommand{\brackets}[1]{\left[#1\right]} % Brackets
\newcommand{\set}[2]{\left\{#1 \; :\; #2\right\}} % Set (two arguments)
\newcommand{\prodin}[2]{\langle #1,#2 \rangle} % Scalar product (two arguments)

\DeclareMathOperator*{\tr}{trace} %Trace
\def\dist{\mathop{\mbox{\normalfont dist}}\limits} % Distance function

\newcommand{\dL}{d\mathcal{L}^{n-1}} % Displays differential of the (n-1)-dimensional Lebesgue measure
\newcommand{\R}{\mathbb R} % Real numbers
\newcommand{\N}{\mathbb N} % Natural numbers
\newcommand{\E}{\mathbb E} % Expectation

\def\Xint#1{\mathchoice
   {\XXint\displaystyle\textstyle{#1}}%
   {\XXint\textstyle\scriptstyle{#1}}%
   {\XXint\scriptstyle\scriptscriptstyle{#1}}%
   {\XXint\scriptscriptstyle\scriptscriptstyle{#1}}%
   \!\int}
\def\XXint#1#2#3{{\setbox0=\hbox{$#1{#2#3}{\int}$}
     \vcenter{\hbox{$#2#3$}}\kern-.5\wd0}}

\def\dashint{\Xint-}

\def\vint_#1{\mathchoice%
          {\mathop{\kern 0.2em\vrule width 0.6em height 0.69678ex depth -0.58065ex
                  \kern -0.8em \intop}\nolimits_{\kern -0.4em#1}}%
          {\mathop{\kern 0.1em\vrule width 0.5em height 0.69678ex depth -0.60387ex
                  \kern -0.6em \intop}\nolimits_{#1}}%
          {\mathop{\kern 0.1em\vrule width 0.5em height 0.69678ex
              depth -0.60387ex
                  \kern -0.6em \intop}\nolimits_{#1}}%
          {\mathop{\kern 0.1em\vrule width 0.5em height 0.69678ex depth -0.60387ex
                  \kern -0.6em \intop}\nolimits_{#1}}}
\def\vintslides_#1{\mathchoice%
          {\mathop{\kern 0.1em\vrule width 0.5em height 0.697ex depth -0.581ex
                  \kern -0.6em \intop}\nolimits_{\kern -0.4em#1}}%
          {\mathop{\kern 0.1em\vrule width 0.3em height 0.697ex depth -0.604ex
                  \kern -0.4em \intop}\nolimits_{#1}}%
          {\mathop{\kern 0.1em\vrule width 0.3em height 0.697ex depth -0.604ex
                  \kern -0.4em \intop}\nolimits_{#1}}%
          {\mathop{\kern 0.1em\vrule width 0.3em height 0.697ex depth -0.604ex
                  \kern -0.4em \intop}\nolimits_{#1}}}

\newcommand{\aveint}[2]{\mathchoice%
          {\mathop{\kern 0.2em\vrule width 0.6em height 0.69678ex depth -0.58065ex
                  \kern -0.8em \intop}\nolimits_{\kern -0.45em#1}^{#2}}%
          {\mathop{\kern 0.1em\vrule width 0.5em height 0.69678ex depth -0.60387ex
                  \kern -0.6em \intop}\nolimits_{#1}^{#2}}%
          {\mathop{\kern 0.1em\vrule width 0.5em height 0.69678ex depth -0.60387ex
                  \kern -0.6em \intop}\nolimits_{#1}^{#2}}%
          {\mathop{\kern 0.1em\vrule width 0.5em height 0.69678ex depth -0.60387ex
                  \kern -0.6em \intop}\nolimits_{#1}^{#2}}}

\newcommand{\ol}{\overline}
\newcommand{\Om}{\Omega}
\newcommand{\I}{\textrm{I}}
\newcommand{\II}{\textrm{II}}

\newcommand{\diam}{\operatorname{diam}}

\title[Tug-of-war games and the $p(x)$-Laplacian] %Use the shortened version of the full title
      {Tug-of-war games with varying probabilities and the normalized $p(x)$-Laplacian}

\author[\'{A}. Arroyo, J. Heino and M. Parviainen]{}


 \email{arroyo@mat.uab.cat}
 \email{joonas.heino@jyu.fi}
 \email{mikko.j.parviainen@jyu.fi}

\begin{document}
\maketitle

% Enter the first author's name and address:
\centerline{\scshape \'{A}ngel Arroyo}
\medskip
{\footnotesize
% please put the address of the first author
 \centerline{Departament de Matem\`{a}tiques, Universitat Aut\`{o}noma de Barcelona}
     \centerline{ 08193 Bellaterra, Barcelona, Spain}
} % Do not forget to end the {\footnotesize by the sign }

\medskip

\centerline{\scshape Joonas Heino and Mikko Parviainen}
\medskip
{\footnotesize
 % please put the address of the second  and third author
 \centerline{ Department of Mathematics and Statistics, University of Jyv\"{a}skyl\"{a}}
   \centerline{ PO Box 35, FI-40014 Jyv\"{a}skyl\"{a}, Finland}

}
\bigskip

 \centerline{}

%The abstract of your paper
\begin{abstract}
We study a two player zero-sum tug-of-war game with varying probabilities that depend on the game location $x$. In particular, we show that the value of the game is locally asymptotically H\"{o}lder continuous. The main difficulty is the loss of translation invariance. We also show the existence and uniqueness of values of the game. As an application, we prove that the value function of the game converges to a viscosity solution of the normalized $p(x)$-Laplacian.
\end{abstract}

\section{Introduction}
The seminal works of Crandall, Evans, Ishii, Lions, Souganidis and others established a  connection between the stochastic differential games and viscosity solution to Bellman-Isaacs equations in the early 80s. However, a similar connection between the %a highly degenerate equations such as
$p$-Laplace or $\infty$-Laplace equations and the tug-of-war games with noise was discovered only rather recently in \cite{peress08,peresssw09}.

In this paper we study a tug-of-war with noise with space dependent probabilities, which is a natural generalization of the original tug-of-war both from mathematical and application point of views. In particular, we prove that the value functions of the game in this setting are asymptotically H\"{o}lder continuous, Theorem \ref{LOCREG}. Here the main difficulty is the loss of translation invariance so that the global or local regularity methods in \cite{peress08}, \cite{manfredipr12} or \cite{luirops13} are not directly applicable. Instead, we employ the method in \cite{luirop}.

 The main idea is to consider two game sequences simultaneously. Heuristically speaking, in a higher dimensional space, the sequences can be linked to a single higher dimensional game by introducing a probability measure that has the measures of the original game as marginals through  suitable couplings.
%For more precise steps of the method, see the beginning of Section \ref{sec:local-reg}.
%
It is interesting to note that couplings of stochastic processes can be employed in the study of regularity for second order linear uniformly parabolic equations with continuous highest order coefficients, see for example  \cite{lindvallr93}, \cite{priolaw06}, and \cite{kusuoka15}. The method has also some similarities to the Ishii-Lions method \cite{ishiil90}, see also \cite{porrettap13}. However, the method we use does not rely on the theorem of sums in the theory of viscosity solutions nor does it use stochastic tools. Indeed, it applies directly to functions satisfying a dynamic programming equation whether they arise from the stochastic games or numerical methods to PDEs.

One of the key tools in studying the tug-of-war games is the dynamic programming principle. For the game in this paper, the dynamic programming principle (DPP) reads as
\begin{equation}\label{DPP}
\begin{split}
				u(x) = ~& \displaystyle\frac{1-\delta(x)}{2}\Bigg[\displaystyle \sup_{\abs\nu=\varepsilon}\pare{\alpha(x)u(x+\nu)+\beta(x)\dashint_{B_\varepsilon^\nu} u(x+h)\dL(h)} \\
				& + \displaystyle \inf_{\abs\nu=\varepsilon}\pare{\alpha(x)u(x+\nu)+\beta(x)\dashint_{B_\varepsilon^\nu} u(x+h)\dL(h)}\Bigg] + \delta(x)F(x)
\end{split}
\end{equation}
with a given boundary cut-off function $\delta$, a boundary function $F$ and probability functions $\alpha(x),\beta(x)$. Here, $B_\varepsilon^\nu$ denotes the $(n-1)$-dimensional ball orthogonal to $\nu$. For more details, see Section \ref{sec:prelim}. Heuristic idea behind the DPP is that the value at a point can be obtained by considering a single step in the game and summing up all the possible outcomes. At the point $x$, the game continues with a probability $1-\delta(x)$. In this case, the maximizer selects the direction $\nu_{\text{max}}$ of fixed radius maximizing the expected payoff at the point. Similarly, the minimizer selects the direction $\nu_{\text{min}}$ of the same radius minimizing the expectation. Then with a probability $\alpha(x)/2$, the game moves to $x+\nu_{\text{max}}$ in the single step, and with the same probability, the game moves to $x+\nu_{\text{min}}$. With a probability $\beta(x)/2$, the next game point is $x+\nu_{\text{max}}'$, where $\nu_{\text{max}}'$ is chosen according to the uniform distribution in a $(n-1)$-dimensional ball orthogonal to $\nu_{\text{max}}$. Similarly with the same probability, the next game point is $x+\nu_{\text{min}}'$, where $\nu_{\text{min}}'$ is chosen uniformly random from a $(n-1)$-dimensional ball orthogonal to $\nu_{\text{min}}$. If the game on the other hand stops at $x$, the payoff is given by the boundary function $F$ at the point.
%The notation $\dashint$ means an average integral, the set $B_\varepsilon^\nu$ is the $(n-1)$-dimensional open ball orthogonal to $\nu$ with radius $\varepsilon>0$ and the measure $\mathcal L^{n-1}$ is the $(n-1)$-dimensional Lebesgue measure.

The first step in the paper is to show that a value function satisfies the dynamic programming principle above and that the value is unique. This is Theorem \ref{value unique}. We first prove existence of a measurable function satisfying the DPP by iterating the operator on the right hand side of (\ref{DPP}). To this end, we guarantee the continuity and thus Borel measurability of the iterands by the boundary correction in the DPP above. Otherwise it is difficult to guarantee the measurability in such iterations.
Then, the uniqueness and the continuity of the solution is obtained by using game theoretic arguments. In particular, we show that the solution coincides with the game value.

As an application, by using the regularity result, Arzel\`{a}-Ascoli's theorem and the DPP, we show in Theorem \ref{theorem:converge} that the values of the game converge to a continuous viscosity solution of the normalized $p(x)$-Laplace equation
 \[
\begin{split}
\Delta_{p(x)}^Nu(x):=\Delta u(x)+(p(x)-2)\Delta^N_\infty u(x)=0,
\end{split}
\]
where $\Delta^N_\infty u:=\abs{\nabla u}^{-2}\sum_{i,j=1}^{n}u_{x_ix_j}u_{x_i}u_{x_j}$ is the normalized infinity Laplacian, and $p:\ol \Om\to (1,\infty)$ is a continuous function on the closure of the game domain $\Omega$ with $\inf_\Omega p>1$ and $\sup_\Omega p<\infty$.
%
%when choosing $\alpha(x)=(p(x)-1)/(p(x)+n),\ \beta(x)=(n+1)/(p(x)+n)$.
Observe that we cover the range $1<p(x)<\infty$. To guarantee that the limit takes the same boundary values, we need boundary estimates which are obtained in Theorem \ref{boundary regularity} by using barrier arguments.

\section{Preliminaries}\label{sec:prelim}
Fix $n\geq 2$ and $\varepsilon>0$ and let $\Omega\subset \mathbb R^n$ be a bounded domain. For measurability reasons, we need the boundary correction function $\delta$  in the dynamic programming principle. Thus, we define the following open sets
\begin{align*}
				I_\varepsilon & =  \{x\in\Omega: \dist(x,\partial\Omega)<\varepsilon\}, \\
				O_\varepsilon & =  \{x\in \mathbb R^n\setminus \ol{\Omega}: \dist(x,\partial\Omega)<\varepsilon\}
\end{align*}
and the set $\Omega_\varepsilon:=\ol{\Omega} \cup O_\varepsilon$. The function $\delta:\ol{\Omega}_\varepsilon\rightarrow[0,1]$ is given by
\begin{equation*}
				\delta(x)=
				\left\{\begin{array}{ll}
								0 & \mbox{ if } x\in\Omega\setminus I_\varepsilon \\
								1-\varepsilon^{-1}\dist(x,\partial\Omega) & \mbox{ if } x\in I_\varepsilon \\
								1 & \mbox{ if } x\in \ol{O}_\varepsilon.
				\end{array}\right.
\end{equation*}
Let $p$ be a continuous function on $\overline\Omega$ satisfying
\begin{equation}\label{function p assump}
				1<p_{\text{min}}:=\inf_{x\in\Omega} p(x)\leq \sup_{x\in\Omega} p(x)=:p_{\text{max}}<\infty.
\end{equation}
We require the finite upper bound $p_{\text{max}}$ to make sure that the tug-of-war game defined below ends almost surely regardless of the strategies. Similarly, the upper bound comes into a play in the techniques we use in Section \ref{sect22}. On the other hand, the regularity and convergence results below require the lower bound in \eqref{function p assump} for the function $p$. To prove existence and uniqueness of continuous solutions to \eqref{DPP} in Section \ref{sec:exist and uniq}, we utilize the uniform continuity of $p$. In Sections \ref{sec:local regularity} and \ref{sec:boundary reg}, the regularity techniques do not require the continuity of $p$, but in Section \ref{sec:applications}, we apply the continuity of $p$.

We define the functions $\alpha,\beta:\overline\Omega\to (0,1)$ depending on $p(x)$ and the dimension $n$ by
\begin{equation*}
				\alpha(x)=\frac{p(x)-1}{p(x)+n} \ \ \mbox{ and } \ \ \beta(x)=1-\alpha(x)=\frac{n+1}{p(x)+n}.
\end{equation*}
By the assumptions on $p(x)$, the functions $\alpha$ and $\beta$ are uniformly continuous. In addition, we have
\begin{equation}\label{alpha assump}
\alpha_{\text{max}}:=\sup_{x\in \Omega}\alpha(x)<1 \ \ \mbox{ and } \ \ \alpha_{\text{min}}:=\inf_{x\in \Omega}\alpha(x)>0.
\end{equation}
We also denote $\beta_{\text{min}}:=1-\alpha_{\text{max}}>0$.

We consider averages of the form
\begin{equation*}
				\dashint_{B_\varepsilon^\nu} u(x+h)\dL(h):=\frac{1}{\mathcal L^{n-1}\big(B_\varepsilon^\nu\big)}\int_{B_\varepsilon^\nu} u(x+h)\dL(h),
\end{equation*}
where $\mathcal{L}^{n-1}$ denotes the $(n-1)$-dimensional Lebesgue measure. The open ball of radius $\varepsilon$ in the $(n-1)$-dimensional hyperplane $\nu^\bot$ orthogonal to $\nu\in \mathbb R^n$ is denoted by $B_\varepsilon^\nu$, i.e.,
\begin{equation*}
				B_\varepsilon^\nu:=B_\varepsilon(0)\cap\nu^\bot:=\{z\in \mathbb R^n:\abs z<\varepsilon \text{ and } \langle z,\nu\rangle =0\}.
\end{equation*}
Throughout the paper, we denote open $n$-dimensional balls of radius $r>0$ by $B_r(x)$ or by $B_r$, if the center point $x\in \mathbb R^n$ plays no role.

For brevity, the compact boundary strip of the game domain is denoted by
$$
\Gamma_{\varepsilon,\varepsilon}:=\ol{I}_\varepsilon\cup \ol{O}_\varepsilon.
$$
Let $F$ be a continuous boundary function $F:\Gamma_{\varepsilon,\varepsilon}\to \mathbb R$. In addition, we define an auxiliary function
\begin{equation}\label{Wu}
				W(x,\nu):=W(u;x,\nu):= \alpha(x)u(x+\nu)+\beta(x)\dashint_{B_\varepsilon^\nu} u(x+h)\dL(h)
\end{equation}
and an operator
\begin{align}
\begin{split}\label{Top}
				T_\varepsilon u(x) := ~& \frac{1-\delta(x)}{2}\Bigg[ \sup_{\abs\nu=\varepsilon}\Big(W(u;x,\nu)\Big)+\inf_{\abs\nu=\varepsilon}\Big(W(u;x,\nu)\Big)\Bigg]  + \delta(x)F(x)
\end{split}
\end{align}
for all $x\in \ol{\Omega}_\varepsilon$ and continuous functions $u\in C(\ol{\Omega}_\varepsilon)$. By using this operator, we can identify the solutions to \eqref{DPP} with the fixed points of $T_\varepsilon$. Note that, despite the fact that $\alpha(x)$ and $\beta(x)$ are not defined in the outside strip $\ol{\Omega}_\varepsilon\setminus \ol{\Omega}$, \eqref{Top} is well-defined by setting $T_\varepsilon u(x)=F(x)$ for all $x\in \ol{\Omega}_\varepsilon\setminus \ol{\Omega}$. Similarly, we set $\delta(x)F(x)=0$ for all $x\in \Omega\setminus I_\varepsilon$.

The same boundary correction as above is also applied in \cite{hartikainen16,luirops14}. For an alternative approach, see \cite{armstrongs12}. Here, this correction is used in order to preserve measurability when iterating the operator. Indeed, in such iterations the measurability can rather easily be lost, see for example \cite[Example 2.4]{luirops14}. In addition, an asymptotic expansion close to \eqref{DPP} is studied in \cite{kawohlmp12}.

\subsection{The two-player tug-of-war game}\label{section2.1}
In this subsection, we introduce the stochastic zero-sum tug-of-war game used in this work. Most of the methods of this paper arise from game theory, and some of the results are even directly proved by using game theory arguments (for example the uniqueness proof in \Cref{LEMMA-UINEQ}).

Let us  consider a game involving two players (say $P_\I$ and $P_{\II}$). A token is placed at a starting point $x_0\in\Omega$. Suppose that, after $j=0,1,2,\ldots$ movements, the token is at a point $x_j\in\Omega$. Then,
\begin{itemize}
\item if $x_j\in\Omega\setminus I_\varepsilon$, $P_\I$ and $P_{\II}$ decide their possible movements $\nu^\I_{j+1}$ and $\nu^{\II}_{j+1}$, respectively, with $\abs{\nu^\I_{j+1}}=\abs{\nu^{\II}_{j+1}}=\varepsilon$. A fair coin is tossed and if $P_i$ wins the toss, we have two possibilities
\begin{itemize}
\item with a probability $\alpha(x_j)$, the token is moved to $x_{j+1}=x_j+\nu^i_{j+1}$, and
\item with a probability $\beta(x_j)$, the token is moved to a point $x_{j+1}\in x_j+B_\varepsilon^{\nu^i_{j+1}}$ uniformly random
\end{itemize}
with $i\in\braces{\I,\II}$.
\item If $x_j\in I_\varepsilon \cup O_\varepsilon$,
\begin{itemize}
\item the game ends with a probability $\delta(x_j)$ and then, $P_{\II}$ pays $P_\I$ the amount given by $F(x_j)$, and
\item with a probability $1-\delta(x_j)$, the players play a game as in the previous case $x_j\in\Omega\setminus I_\varepsilon$.
\end{itemize}
\end{itemize}
Let $\tau$ denote the time when the game ends, and denote by $x_\tau\in \Gamma_{\varepsilon,\varepsilon}$ the position where the game ends. Then, $P_{\II}$ pays $P_\I$ the quantity $F(x_\tau)$. 

We can construct the game described above by the following procedure. Let $(c_j)_{j=0}^\infty$ be a sequence of random variables such that $c_j\in \{0,1\}=:\widetilde{C}$ for all $j\geq 0$ with $c_0:=0$. The random variable $c_j$ gives the information whether the $j$th movement of the game has been decided by playing the game. If it holds $c_j=0$, the position $x_j$ is selected by playing the game. On the other hand, if it holds $c_j=1$, we have $x_j=x_{j-1}$.

Let $\xi_0,\xi_1,\xi_2,\dots$ be independent and identically distributed random variables such that $\xi_0$ is distributed uniformly random on $[0,1]$. Moreover, the process $(\xi_j)_{j=0}^\infty$ is independent of the game process $(x_j)_{j=0}^\infty$. Then for all $j\geq1$, the  probability distribution of the random variable $c_{j}$ is determined by 
\begin{align*}
c_{j}&=\begin{cases}
0,&\text{if } \xi_{j-1}\leq1-\delta(x_{j-1}), \\
1,&\text{if } \xi_{j-1}>1-\delta(x_{j-1})
\end{cases}
\end{align*}
given that it holds $c_{j-1}=0$. If it holds $c_{j-1}=1$, then we have $c_j=1$. By this definition of $c_j$ for all $j\geq 1$, we can define the random variable $\tau$ by
\begin{equation}\label{rv tau}
\tau=\inf_{j\geq 0}\big\{c_{j+1}=1\big\}.
\end{equation}

We define a \textit{history} of the game as the vector $\big((c_0,x_0),(c_1,x_1),\ldots,(c_j,x_j)\big)$ describing the positions of the token and the information whether the positions had been taken by playing the game at each step after $j$ repetitions. A \textit{strategy} is a sequence of Borel measurable functions that gives the next game position given the history of the game. Therefore, we define $\mathcal{S}_i:=\mathcal (\mathcal S_i^j)_{j=1}^\infty$ with
\begin{equation*}
				\mathcal{S}_i^j:\big\{(c_0,x_0)\big\} \times \bigcup_{k=1}^{j-1}(\widetilde{C}\times\Omega_\varepsilon)^k\to \partial B_\varepsilon(0)
\end{equation*}
for all $j\in \mathbb N$ and with both $i\in\braces{\I,\II}$. For example, we have for $P_\I$ and for all $j\geq 1$ that
$$
\mathcal{S}_\I^j\Big(\big((c_0,x_0),\dots,(c_{j-1},x_{j-1})\big)\Big)=\nu_j^\I\in \partial B_\varepsilon(0).
$$
Given a starting point $x_0\in\Omega$ and strategies $\mathcal{S}_\I,\mathcal{S}_{\II}$, we define a probability measure $\mathbb{P}_{\mathcal{S}_\I,\mathcal{S}_{\II}}^{x_0}$ on the natural product $\sigma$-algebra of the space of all game trajectories. This measure is built by applying Kolmogorov's extension theorem to the family of transition densities
\begin{align*}
&\pi_{\mathcal{S}_\I,\mathcal{S}_{\II}}\big((c_0,x_0),(c_1,x_1),\dots,(c_{j},x_j);C,A \big) \\
&=\big(1-\delta(x_j)\big)\pi_{\mathcal{S}_\I,\mathcal{S}_{\II}}^{\text{local}}\big((x_0,x_1,\dots,x_j);A \big)\mathbb I_0(C)\mathbb I_{c_j}\big(\{0\}\big) \\
&\hspace{1em}+\delta(x_j)\mathbb I_{x_j}(A)\mathbb I_1(C)\mathbb I_{c_j}\big(\{0\}\big)+\mathbb I_{c_j}\big(\{1\}\big)\mathbb I_{x_{j}}(A)
\end{align*}
for all Borel subsets $A\subset\mathbb R^n$ and $C\subset \widetilde{C}$. With a slight abuse of the notation, for all points $z$ and sets $B$, the measure $\mathbb I_z(B)$ is one, if $z\in B$, and zero otherwise. Moreover, it holds
\begin{align*}
&\pi^{\text{local}}_{\mathcal{S}_\I,\mathcal{S}_{\II}}\big((x_0,\dots,x_{j});A \big)=\frac{1}{2}\bigg[\alpha(x_j)\Big(\mathbb I_{x_j+\nu_{j+1}^\I}(A)+\mathbb I_{x_j+\nu_{j+1}^\II}(A)\Big) \\
&+\frac{\beta(x_j)}{\omega_{n-1}\varepsilon^{n-1}}\bigg(\mathcal L^{n-1}\Big(B_\varepsilon^{\nu_{j+1}^\I}(x_j)\cap A\Big)+\mathcal L^{n-1}\Big(B_\varepsilon^{\nu_{j+1}^\II}(x_j)\cap A\Big)\bigg)\bigg]
\end{align*}
with the constant $\omega_{n-1}:=\mathcal L^{n-1}\big(B_1^z\big)$ for any $z\in \mathbb R^n\setminus \{0\}$. Furthermore, we denote $B_\varepsilon^z(y):=y+B_\varepsilon^z$ for $z\in \mathbb R^n \setminus \{0\}$ and $y\in \mathbb R^n$. 

Here, we follow the ideas from \cite{hartikainen16}, where the constant $\alpha$ case is covered. For the benefit of the reader and since the setting is slightly different, we give a self-contained proof.
\begin{lemma}\label{gamestop}
The game ends almost surely in finite time regardless of the strategies $\mathcal{S}_\I$ and $\mathcal{S}_{\II}$.
\end{lemma}

\begin{proof}
The idea of the proof is to consider solely random movements and to find a uniform lower bound for the probability of the event that the modulus of $\abs{x_j}$ grows in a suitable fashion. In the proof, we need the fact $\beta_{\text{min}}>0$.

Let $x_0\in\Omega$, $j\geq 0$ and let $x_{j+1}=x_j+h_j$, where $h_j$ represents the displacement at each step of the game. By the vector calculus, we have
\begin{equation*}
				\abs{x_{j+1}}^2 = \abs{x_j}^2+\abs{h_j}^2+2\prodin{x_j}{h_j}.
\end{equation*}
In addition by the definition of the game, $h_j$ is randomly chosen from $B_\varepsilon^{\nu}$ with a probability $\beta(x_j)/2$ for the vector $\nu:=\nu_{j+1}^\I$. Moreover, given that a random movement is chosen from $B_\varepsilon^\nu$, we have $\prodin{x_j}{h_j}\geq 0$ with a probability of at least $\frac{1}{2}$ and the event $\abs{h_j}\geq\frac{\varepsilon}{2}$ has a probability of
$$
1-\frac{\mathcal L^{n-1}\big(B_{\varepsilon/2}^\nu\big)}{\mathcal L^{n-1}\big(B_\varepsilon^\nu\big)}=1-2^{1-n}.
$$
Consequently, there is a positive probability of a random movement $h_j$ such that $|h_j|\geq \varepsilon/2$ and $\prodin{x_j}{h_j} \geq 0$. In this case, we have
\begin{equation}\label{growth of modulus}
				\abs{x_{j+1}}^2 \geq \abs{x_j}^2+\frac{\varepsilon^2}{4}
\end{equation}
with a probability of at least
\begin{equation*}
				\beta(x_j)\pare{\frac{1}{4}-\frac{1}{2^{n+1}}}\geq\beta_{\text{min}}\pare{\frac{1}{4}-\frac{1}{2^{n+1}}}=:\theta>0.
\end{equation*}
Note that the universal constant $\theta$ does not depend on $j$ and the fact $\beta_{\text{min}}>0$ implies $\theta>0$. Now, let
\begin{equation*}
				j_0:=j_0(\varepsilon,\Omega)=4\left\lceil\diam(\Omega)\varepsilon^{-2}\right\rceil\in\N.
\end{equation*}
Then, after $j_0$ consecutive movements in the way \eqref{growth of modulus} we have
\begin{equation*}
				\abs{x_{j_0}}^2\geq\abs{x_0}^2+j_0\frac{\varepsilon^2}{4}>|x_0|^2+\diam(\Omega).
\end{equation*}
Therefore, the token has exited the game domain after at most $j_0$ steps for any starting point $x_0$ with a probability of at least $\theta^{j_0}$. Consequently, the probability of not exiting the game domain after $j_0$ steps is bounded above by $1-\theta^{j_0}$.

By repeating $kj_0$ times the game, the probability of not exiting $\Omega$ after $kj_0$ steps is bounded above by
\begin{equation*}
				\left(1-\theta^{j_0}\right)^k.
\end{equation*}
Thus, by letting $k\rightarrow\infty$, this probability goes to zero, and the proof is completed.
\end{proof}

For all starting points $x_0\in\Omega$, we define a \textit{value function} for $P_\I$ and for $P_\II$ by
\begin{equation}\label{uIIIdef}
				\left\{\begin{array}{l}
								u_\I(x_0)=\displaystyle\sup_{\mathcal{S}_\I}\inf_{\mathcal{S}_{\II}}\E_{\mathcal{S}_\I,\mathcal{S}_{\II}}^{x_0}[F(x_\tau)],\\ 
								u_{\II}(x_0)=\displaystyle\inf_{\mathcal{S}_{\II}}\sup_{\mathcal{S}_\I}\E_{\mathcal{S}_\I,\mathcal{S}_{\II}}^{x_0}[F(x_\tau)].
				\end{array}\right.
\end{equation}

\section{Existence and uniqueness}\label{sec:exist and uniq}
In this section, the goal is to prove that there exists a unique continuous solution satisfying the dynamic programming principle \eqref{DPP}. The proof is divided into two parts. In \Cref{sect21}, by iterating the operator $T_\varepsilon$ defined in \eqref{Top}, we show that there exist a lower and an upper semicontinuous solution to \eqref{DPP}. Then in \Cref{sect22}, we show that every measurable solution to \eqref{DPP} is bounded between the lower and the upper semicontinuous solutions. Further, we prove by using the tug-of-war game defined in \Cref{section2.1} that, in fact, both semicontinuous solutions are the same.

\subsection{Existence of semicontinuous solutions to \eqref{DPP}}\label{sect21}

In this subsection, by iterating the operator $T_\varepsilon$, we construct monotone sequences of bounded continuous functions. As a consequence, these sequences converge to semicontinuous functions which turn out to be solutions to \eqref{DPP}. With that purpose, first, we need to show that $T_\varepsilon$ maps continuous functions into continuous functions.

\begin{lemma}\label{LEMMA-WCONT}
For any continuous function $u\in C(\overline\Omega_\varepsilon)$, the function $W(x,\nu)$ defined in \eqref{Wu} is continuous with respect to each variable on $\ol{\Omega} \times \partial B_\varepsilon(0)$.
\end{lemma}

\begin{proof}
For fixed $\abs{\nu}=\varepsilon$, we have for any $x,y \in \ol{\Omega}$ the estimate
\begin{align*}			
&\abs{\alpha(x)u(x+\nu)-\alpha(y)u(y+\nu)}\\
&\leq |\alpha(x)u(x+\nu)-\alpha(x)u(y+\nu)| +|\alpha(x)u(y+\nu)-\alpha(y)u(y+\nu)|\\
&\leq \alpha(x)\omega_u(\abs{x-y})+\norm{u}_\infty\omega_\alpha(\abs{x-y}),
\end{align*}
where $\omega_f$ is a modulus of continuity of the uniformly continuous function $f$. In a similar way, we have
\begin{multline*}
				\abs{\beta(x)\dashint_{B_\varepsilon^\nu} u(x+h)\dL(h)-\beta(y)\dashint_{B_\varepsilon^\nu} u(y+h)\dL(h)} \\
				\leq \beta(x)\omega_u(\abs{x-y})+\norm{u}_\infty\omega_\beta(\abs{x-y})
\end{multline*}
for $x,y \in \ol{\Omega}$. Thus, these inequalities imply that
\begin{equation}\label{estW}
				\abs{W(x,\nu)-W(y,\nu)} \leq \omega_u(\abs{x-y})+\norm u_\infty\Big[\omega_\alpha(\abs{x-y})+\omega_\beta(\abs{x-y})\Big]
\end{equation}
for all $x,y\in \ol{\Omega}$. Hence, $W(\cdot,\nu)$ is a continuous function for fixed $\nu$ with modulus of continuity $\omega_u+\norm u_\infty\big[\omega_\alpha+\omega_\beta\big]$.

For the continuity on $\nu$, fix a point $x\in\ol{\Omega}$. Then, the modulus of continuity of $\alpha$ does not play any role. In addition, since the function $u$ is continuous by the hypothesis, we only need to check the continuity of the function
\begin{equation*}
				\nu \mapsto \dashint_{B_\varepsilon^\nu} u(x+h)\dL(h).
\end{equation*}
Let $\abs\nu=\abs\chi=\varepsilon$ and define a rotation $P:\nu^\bot\rightarrow\chi^\bot$ satisfying
\begin{equation}\label{rotation}
				\abs{h-Ph}\leq C\abs{h}\abs{\nu-\chi}
\end{equation}
for all $h\in \nu^\bot$, where $C>0$ is a constant not depending on the choices of $\nu$ and $\chi$. Therefore, we have
\begin{multline}\label{eqrot}
				\dashint_{B_\varepsilon^\nu} u(x+h)\dL(h)-\dashint_{B_\varepsilon^\chi} u(x+h)\dL(h) \\
				= \dashint_{B_\varepsilon^\nu}\brackets{u(x+h)-u(x+Ph)}\dL(h).
\end{multline}
By recalling \eqref{rotation} together with the fact that we can choose $\omega_u$ to be increasing, we can estimate the expression in brackets in the equation \eqref{eqrot} from above by
\begin{equation*}
				\omega_u\pare{C\varepsilon\abs{\nu-\chi}}
\end{equation*}
for $h\in B_\varepsilon^\nu$. Then, this same bound also holds for \eqref{eqrot}, and the continuity of $W(x,\cdot)$ for fixed $x\in \ol{\Omega}$ follows.
\end{proof}

\begin{lemma}\label{LEMMA-TLIP}
For all $u\in C(\overline\Omega_\varepsilon)$, the operator $T_\varepsilon$ defined in \eqref{Top} satisfies $T_\varepsilon u\in C(\overline\Omega_\varepsilon)$. In addition, for all $u,v\in C(\overline\Omega_\varepsilon)$ such that $u\leq v$, we have
$$
T_\varepsilon u\leq T_\varepsilon v~ (\textit{monotonicity}).
$$
\end{lemma}

\begin{proof}
The monotonicity of $T_\varepsilon$ follows easily from the definition \eqref{Top}. Let $u\in C(\overline\Omega_\varepsilon)$ be a function with a modulus of continuity $\omega_u$. By \eqref{Top} and the fact that $F$ is continuous on $\overline O_\varepsilon$, the function $T_\varepsilon u$ is continuous on the outside strip $\overline O_\varepsilon$. Thus, we have to check that $T_\varepsilon u$ is continuous on $\ol{\Omega}$.

First, let $x,y\in \Omega\setminus I_\varepsilon$ and recall the elementary inequalities
\begin{align*}
				\abs{\sup_{\abs\nu=\varepsilon}W(x,\nu)-\sup_{\abs\nu=\varepsilon}W(y,\nu)} & \leq  \sup_{\abs\nu=\varepsilon}\abs{W(x,\nu)-W(y,\nu)},\\
				\abs{\inf_{\abs\nu=\varepsilon}W(x,\nu)-\inf_{\abs\nu=\varepsilon}W(y,\nu)} & \leq  \sup_{\abs\nu=\varepsilon}\abs{W(x,\nu)-W(y,\nu)}.
\end{align*}
Then by the inequality \eqref{estW} for any $\abs\nu=\varepsilon$, we get that
\begin{align*}
&\frac{1}{2}\abs{\pare{\sup_{\abs\nu=\varepsilon}+\inf_{\abs\nu=\varepsilon}}W(x,\nu)-\pare{\sup_{\abs\nu=\varepsilon}+\inf_{\abs\nu=\varepsilon}}W(y,\nu)} \\
&\leq \omega_u(\abs{x-y})+\norm u_\infty\Big[\omega_\alpha(\abs{x-y})+\omega_\beta(\abs{x-y})\Big].
\end{align*}
Here, we use the shorthand notation
$$
\pare{\sup_{\abs\nu=\varepsilon}+\inf_{\abs\nu=\varepsilon}}W(x,\nu):=\sup_{\abs\nu=\varepsilon}W(x,\nu)+\inf_{\abs\nu=\varepsilon}W(x,\nu).
$$
Therefore, since $\delta=0$ on $\Omega\setminus I_\varepsilon$, we have shown that $T_\varepsilon u$ is continuous on $\Omega\setminus I_\varepsilon$.

Then, let $x,y \in I_\varepsilon$ and recall that $\sup_{x\in\Omega}\big( 1-\delta(x)\big)=1$ and $\omega_{\delta}(t)=t/\varepsilon$ for $t\geq 0$. Thus, we can estimate
\begin{align*}			
				&\abs{\frac{1-\delta(x)}{2}\pare{\sup_{\abs\nu=\varepsilon}+\inf_{\abs\nu=\varepsilon}}W(x,\nu)-\frac{1-\delta(y)}{2}\pare{\sup_{\abs\nu=\varepsilon}+\inf_{\abs\nu=\varepsilon}}W(y,\nu)}\\
				&\leq \omega_u(\abs{x-y})+\norm u_\infty\big[\omega_\alpha(\abs{x-y})+\omega_\beta(\abs{x-y})\big]+\frac{\norm{u}_\infty}{\varepsilon}\abs{x-y}
\end{align*}
and
\begin{equation*}			
				\abs{\delta(x)F(x)-\delta(y)F(y)} \leq \delta(x)\omega_F(\abs{x-y})+\frac{\norm{F}_\infty}{\varepsilon}\abs{x-y}.
\end{equation*}
Consequently, $T_\varepsilon u$ is continuous in $I_\varepsilon$. Since the limiting values of the function $T_\varepsilon u$ coincide with the function values on the boundary $\partial I_\varepsilon$, there must exist a modulus of continuity for $T_\varepsilon u$, and hence $T_\varepsilon u\in C(\overline\Omega_\varepsilon)$.
\end{proof}

For the next result, let $T_\varepsilon^k$ denote the $k$-th iteration of the operator $T_\varepsilon$ for $k\in\N$, i.e.,
\begin{equation*}
								T_\varepsilon^k=T_\varepsilon(T_\varepsilon^{k-1}), \mbox{ } T_\varepsilon^0=\text{Id},
\end{equation*}
with the identity operator $\text{Id}(u)=u$ for all $u\in C(\ol{\Omega}_\varepsilon)$. By \eqref{Top} and the monotonicity of $T_\varepsilon$, the sequence of iterates $\braces{T_\varepsilon^k(\inf F)}_k$ is increasing and $\braces{T_\varepsilon^k(\sup F)}_k$ is decreasing. Moreover,
\begin{equation}\label{seqite}
				\inf F\leq T_\varepsilon^k(\inf F)\leq T_\varepsilon^k(\sup F)\leq\sup F
\end{equation}
for all $k\in\N$. Consequently, we can define the pointwise limit of both sequences
\begin{equation}\label{smcdef}
				\left\{\begin{array}{l}
								\underline u(x):=\displaystyle\lim_{k\rightarrow\infty}T^k_\varepsilon(\inf F),\\ 
								\overline u(x):=\displaystyle\lim_{k\rightarrow\infty}T^k_\varepsilon(\sup F)
				\end{array}\right.
\end{equation}
for all $x\in\ol{\Omega}_\varepsilon$. In addition, since $\underline u$ and $\overline u$ are defined as the limit of monotone sequences of continuous functions, they are lower and upper semicontinuous functions, respectively.

\begin{proposition}\label{LEMMA-SMC}
The functions $\underline u$ and $\overline u$ defined in \eqref{smcdef} are solutions to \eqref{DPP} and satisfy
\begin{equation}\label{smc}
				\underline u\leq\overline u.
\end{equation}
\end{proposition}

\begin{proof}
The inequality \eqref{smc} follows easily from \eqref{seqite}. We only show that $\underline u$ is a solution to \eqref{DPP}, since a similar argument can be applied to $\overline u$. To establish the result, we use \Cref{LEMMA-WCONT,LEMMA-TLIP} and the fact that $\{T^k_\varepsilon(\inf F)\}$ is increasing to show that we can change the order of the limit and the infimum in the function $\underline u$.

Let $x\in \ol{\Omega}_\varepsilon$ and $u_k:=T_\varepsilon^k(\inf F)$ for $k\in\N$. Then,
\begin{equation*}
\begin{split}
				\underline u(x) = ~& \lim_{k\rightarrow\infty}u_{k+1}(x) 
				=  \lim_{k\rightarrow\infty}T_\varepsilon u_k(x) \\
				= ~& \frac{1-\delta(x)}{2}\Bigg[\lim_{k\rightarrow\infty}\sup_{\abs\nu=\varepsilon}W(u_k;x,\nu) +\lim_{k\rightarrow\infty}\inf_{\abs\nu=\varepsilon}W(u_k;x,\nu)\Bigg] 
				 + \delta(x)F(x),
\end{split}
\end{equation*}
where $W$ denotes the auxiliary function defined in \eqref{Wu}. Thus, we need to prove the equalities
\begin{equation*}
				\lim_{k\rightarrow\infty}\sup_{\abs\nu=\varepsilon}W(u_k;x,\nu) = \sup_{\abs\nu=\varepsilon}W(\underline{u};x,\nu)
\end{equation*}
and
\begin{equation*}
				\lim_{k\rightarrow\infty}\inf_{\abs\nu=\varepsilon}W(u_k;x,\nu) = \inf_{\abs\nu=\varepsilon}W(\underline{u};x,\nu).
\end{equation*}
The first equation follows from the fact that the sequence $\braces{u_k}$ is pointwise increasing. For the second equation, we can assume $x\in \Omega$. \Cref{LEMMA-WCONT,LEMMA-TLIP} imply that $W(u_k;x,\nu)$ is continuous with respect to $\nu$ for all $k\geq 1$. Therefore, we can define the compact set
\begin{equation*}
				C_k(\lambda):=\{\nu \in \mathbb R^n:|\nu|=\varepsilon \text{ and }W(u_k;x,\nu)\leq\lambda\}
\end{equation*}
for $\lambda\in\R$. Again, since $\braces{u_k}$ is pointwise increasing, $C_{k+1}(\lambda)\subset C_k(\lambda)$ for all $k\geq 1$. Now, let
\begin{equation*}
				\lambda=\lim_{k\rightarrow\infty}\inf_{\abs\nu=\varepsilon}W(u_k;x,\nu).
\end{equation*}
Because $W(u_k;x,\cdot)$ is continuous for all $k\geq 1$, there exists $\nu^*_k\in \partial B_\varepsilon(0)$ such that
$$
\inf_{\abs\nu=\varepsilon}W(u_k;x,\nu)=W(u_k;x,\nu^*_k).
$$
This, together with the fact that $\{u_k\}$ is increasing, yields $C_k(\lambda)\neq\emptyset$ for all $k\geq 1$. Thus by Cantor's intersection theorem, we get
\begin{equation*}
				\bigcap_{k=1}^\infty C_k(\lambda)\neq\emptyset.
\end{equation*}
Choose $\tilde{\nu}\in\cap_{k=1}^\infty C_k(\lambda)$ so that we can estimate
\begin{equation*}
				\lambda\leq \inf_{\abs\nu=\varepsilon}W(\underline{u};x,\nu)\leq W(\underline{u};x,\tilde{\nu}) = \lim_{k\rightarrow\infty} W(u_k;x,\tilde{\nu}) \leq \lambda.
\end{equation*}
The first inequality  follows from the choice of $\lambda$ and the fact that $\braces{u_k}$ is increasing. In addition, we use the monotone convergence theorem in the first equality and the choice of $\tilde{\nu}$ in the last inequality. Therefore, the proof is complete.
\end{proof}

\subsection{Uniqueness of solutions to \eqref{DPP}}\label{sect22}
In this subsection, we prove the uniqueness of solutions to \eqref{DPP}. To establish the result, we first show that any measurable solution of the equation \eqref{DPP} is between the solutions $\underline u$ and $\ol{u}$. Then, we show that, in fact, the functions $\underline u$ and $\ol{u}$ coincide. For the first result, we need the following technical lemma.

\begin{lemma}\label{LEMMA-MEASSOL}

Let $u$ be a measurable solution to \eqref{DPP}. Assume that $\sup F<\sup_{\Omega} u$, and let $x\in\Omega$ be such that
\begin{equation}\label{counterA}
				u(x)>\max\braces{\sup F, \sup_{\Omega} u-\lambda}
\end{equation}
for $\lambda>0$. Then, there exist $\abs{\nu_0}=\varepsilon$ and $h_0\in B_\varepsilon^{\nu_0}$ satisfying the inequalities
\begin{equation}\label{ineqA01}
				\abs{x+h_0}^2\geq \abs{x}^2+2^{\frac{2}{1-n}}\varepsilon^2
\end{equation}
and
\begin{equation}\label{ineqA02}
				u(x+h_0) \geq \sup_{ \Omega}  u - c(\alpha)\lambda
\end{equation}
with a constant $c(\alpha)>1$.

\begin{proof}
We obtain the inequalities \eqref{ineqA01} and \eqref{ineqA02} by analyzing the dynamic programming principle \eqref{DPP}. The proof is similar to the proof of Lemma \ref{gamestop}. Since $u$ satisfies \eqref{DPP}, we have
\begin{align*}
				u(x) \leq ~& (1-\delta(x))\displaystyle\sup_{\abs\nu=\varepsilon}\pare{\alpha(x)u(x+\nu)+\beta(x)\dashint_{B_\varepsilon^\nu} u(x+h)\dL(h)}  +\delta(x)\sup F \\
								\leq ~& (1-\delta(x))\alpha(x)\sup_{\abs\nu=\varepsilon}u(x+\nu) +(1-\delta(x))\beta(x)\sup_{\abs\nu=\varepsilon}\dashint_{B_\varepsilon^\nu} u(x+h)\dL(h) \\
				~& +\delta(x)\sup F.
\end{align*}
In addition, by utilizing the assumption \eqref{counterA} and $u=F$ on $\ol{O}_\varepsilon$, we get
$$
\sup_{\abs\nu=\varepsilon}u(x+\nu)\leq u(x)+\lambda.
$$
Thus by \eqref{counterA}, $0<\delta(x)<1$, $\alpha(x)+\beta(x)=1$ and $\beta(x)\geq\beta_{\text{min}}>0$ for all $x\in \Omega$, we have
\begin{equation*}
				u(x)\leq\sup_{\abs\nu=\varepsilon}\dashint_{B_\varepsilon^\nu} u(x+h)\dL(h)+\frac{\alpha_{\text{max}}}{\beta_{\text{min}}}\lambda.
\end{equation*}
By the definition of supremum, there must exist $\abs{\nu_0}=\varepsilon$ such that
\begin{equation}\label{eqn01}
				u(x)-2\frac{\alpha_{\text{max}}}{\beta_{\text{min}}}\lambda\leq\dashint_{B_\varepsilon^{\nu_0}} u(x+h)\dL(h).
\end{equation}

Next, we define a set $S\subset B_\varepsilon^{\nu_0}$ depending on $x$ and $\nu_0$. If $ x\not=0$ and $\nu_0 \in \text{span}\{x\}$ or $x=0$, we define
\begin{align*}
S:=\set{h\in B_\varepsilon^{\nu_0}}{\abs{h}\geq (3/4)^{\frac{1}{n-1}}\varepsilon}.
\end{align*}
Otherwise, we set
$$
S:=\set{h\in B_\varepsilon^{\nu_0}}{\abs{h}\geq 2^{\frac{1}{1-n}}\varepsilon \mbox{ and } \prodin{x}{h}\geq 0}.
$$
Observe that in both cases, the Lebesgue measure of the set $S$ is the same. Indeed, it is clear that
$$
\mathcal{L}^{n-1}(B_\varepsilon^{\nu_0})-\mathcal{L}^{n-1}(B_{(3/4)^{1/(n-1)}\varepsilon}^{\nu_0})=\frac{1}{4}\mathcal{L}^{n-1}(B_\varepsilon^{\nu_0}).
$$
By symmetry, we get
$$
\mathcal L^{n-1}\big(\{h\in B_\varepsilon^{\nu_0}: \langle x,h\rangle>0\}\big)=\mathcal L^{n-1}\big(\{h\in B_\varepsilon^{\nu_0}: \langle x,h\rangle<0\}\big),
$$
and in the case $x\not=0$ and $\nu_0 \not \in \text{span}\{x\}$, it holds $$\mathcal L^{n-1}\big(\{h\in B_\varepsilon^{\nu_0}: \langle x,h\rangle=0\}\big)=0.$$
Thus, we have
\begin{equation}\label{eq:leb of S}
	 		\mathcal{L}^{n-1}(S)=\frac{1}{4}\mathcal{L}^{n-1}(B_\varepsilon^{\nu_0}).
\end{equation}
In addition, because $(3/4)^{\frac{2}{n-1}}\geq  2^{\frac{2}{1-n}}$, the inequality \eqref{ineqA01} holds for each $h\in S$. The equality \eqref{eq:leb of S}, together with \eqref{counterA} and \eqref{eqn01}, implies
\begin{equation*}
\begin{split}
				\sup_{\Omega}u  \leq ~& u(x)+\lambda \\
				\leq ~& \dashint_{B_\varepsilon^{\nu_0}}u(x+h)\dL(h) + \lambda\frac{1+\alpha_{\text{max}}}{\beta_{\text{min}}} \\
				= ~& \frac{1}{4\mathcal{L}^{n-1}(S)}\braces{\int_{S}u(x\!+\!h)\dL(h)\!+\!\int_{B_\varepsilon^{\nu_0}\setminus S}u(x\!+\!h)\dL(h)}  \!+\!\lambda\frac{1\!+\!\alpha_{\text{max}}}{\beta_{\text{min}}} \\
				\leq ~& \frac{1}{4}\ \dashint_{S}u(x+h)\dL(h) + \frac{3}{4}\sup_{\Omega} u +\frac{2\lambda}{\beta_{\text{min}}}.
\end{split}
\end{equation*}
By rearranging the terms and multiplying by $4$, we have
\begin{equation*}
\begin{split}
				\sup_{\Omega}u - \frac{8\lambda}{\beta_{\text{min}}} \leq \dashint_{S}u(x+h)\dL(h).
\end{split}
\end{equation*}
Hence, there must exist $h_0\in S\subset B_\varepsilon^{\nu_0}$ satisfying \eqref{ineqA02}.
\end{proof}
\end{lemma}

\begin{proposition}\label{LEMMA-MEASSOL1}
Any measurable solution $u$ to \eqref{DPP} satisfies
\begin{equation*}
				\underline u\leq u\leq\overline u
\end{equation*}
with $\underline u$ and $\overline u$ the semicontinuous functions defined in \eqref{smcdef}.
\end{proposition}

\begin{proof}
By the monotonicity of the operator $T_\varepsilon$ and the definitions of $\underline u$ and $\overline u$, it is enough to show that
\begin{equation}\label{MEASSOL-INEQ}
				\inf F\leq u\leq\sup F.
\end{equation}
Because $u$ is a solution to \eqref{DPP}, we have $u(x)=F(x)$ for $x\in \ol{O}_\varepsilon$. Hence, we need to show the estimate \eqref{MEASSOL-INEQ} for all $x \in \Omega$. We focus our attention on the second inequality, since the proof of the first inequality is analogous. We proceed by contradiction and assume that
\begin{equation*}
				\sup_{\Omega} u>\sup F.
\end{equation*}
By the assumption, for $\eta>0$ there exists a point $x_1\in\Omega$ such that
\begin{equation*}
				u(x_1)>\max\braces{\sup F,\sup_{\Omega}u -\eta}.
\end{equation*}

The idea of the proof consists of finding a sequence of points $\braces{x_j}$ satisfying $u(x_j)>\sup F$ for all $j$ and $|x_{j_0}|$ is big enough for some large $j_0\geq 1$. This is a contradiction, because $u=F$ on $\ol{O}_\varepsilon$. We obtain the sequence of points by using Lemma \ref{LEMMA-MEASSOL} iteratively.

Choose an integer $j_0:=j_0(\varepsilon,n,\Omega)\geq 1$ big enough such that
\begin{equation}\label{xseq}
j2^{\frac{2}{1-n}}\varepsilon^2>\diam(\Omega)
\end{equation}
for all $j\geq j_0$. Then, we fix the constant $\eta>0$ small enough such that
\begin{equation}\label{contant eta small}
				0<\eta<\frac{1}{c(\alpha)^{j_0}}\pare{\sup_{ \Omega} u-\sup F}
\end{equation}
with the constant $c(\alpha)>1$ from \Cref{LEMMA-MEASSOL}. We start from $x_1$ and choose $x_2$ such that $x_2=x_1+h_0$ with $h_0$ given by \Cref{LEMMA-MEASSOL}. Then, we have that $|x_2|^2\geq |x_1|^2+2^{\frac{2}{1-n}}\varepsilon^2$ and
$$
u(x_2)\geq \sup_{\Omega} u-c(\alpha)\eta>\sup F.
$$
If $x_2\in \ol{O}_\varepsilon$, we get a contradiction. Otherwise, we continue in the same way. We choose $x_3$ such that $x_3=x_2+h_1$ with $h_1$ given by \Cref{LEMMA-MEASSOL}. Then, we have that $|x_3|^2\geq |x_1|^2+2\cdot2^{\frac{2}{1-n}}\varepsilon^2$ and
$$
u(x_3)\geq \sup_{\Omega} u-c(\alpha)^2\eta>\sup F.
$$

After $j_0-1$ repetitions, assume that $x_{j_0}\in \Omega$. By the inequalities \eqref{xseq} and \eqref{contant eta small} it holds for the point $x_{j_0+1}$ that
\begin{equation*}
				\abs{x_{j_0+1}}^2\geq\abs{x_1}^2+j_02^{\frac{2}{1-n}}\varepsilon^2>\abs{x_1}^2+\diam(\Omega)
\end{equation*}
and
\begin{equation*}
				u(x_{j_0+1})\geq \sup_{\Omega} u-c(\alpha)^{j_0}\eta>\sup F.
\end{equation*}
Since $x_{j_0+1}\not \in \Omega$, the contradiction follows.
\end{proof}

The next theorem, together with \eqref{smc}, implies that the semicontinuous solutions to \eqref{DPP}, $\underline u$ and $\overline u$, coincide.
\begin{theorem}\label{LEMMA-UINEQ}
Let $\underline u$ and $\overline u$ be the semicontinuous functions defined in \eqref{smcdef}. In addition, let $u_\I$ and $u_{\II}$ be the value functions defined in \eqref{uIIIdef}. Then, we have that
\begin{equation*}
				\overline u \leq u_\I \leq u_{\II}\leq \underline u.
\end{equation*}
\end{theorem}

\begin{proof}
From the properties of $\inf$ and $\sup$, it is clear that
\begin{equation*}
				u_\I\leq u_{\II}.
\end{equation*}
Thus, we need to prove that
\begin{equation*}
u_{\II}\leq\underline u \ \ \mbox{ and } \ \ \overline u\leq u_\I.
\end{equation*}

We only show $u_{\II}\leq\underline u$, since the argument in the other case is similar. To establish the result, we find a suitable strategy for $P_{\II}$ and a function $(c,x)\mapsto \Phi(c,x)$ depending on $\underline u$ and $F$ such that the process $\Phi(c_k,x_k)$ becomes a supermartingale irregardless what the opponent does. Then, we will be able to compare the functions $\underline u$ and $u_{\II}$ by the optimal stopping theorem.

Let $x_0 \in \Omega$ and denote a strategy $\mathcal{S}^*_{\II}$ for $P_{\II}$ such that
%\begin{multline*}
%				\alpha(x_j)\underline u(x_j+\nu_j^{\II})+\beta(x_j)\dashint_{B_\varepsilon^{\nu_j^{\II}}}\underline u(x_j+h)\dL(h) \\
%				=\inf_{\abs\nu=\varepsilon}\pare{\alpha(x_j)\underline u(x_j+\nu)+\beta(x_j)\dashint_{B_\varepsilon^{\nu}}\underline u(x_j+h)\dL(h)}
%\end{multline*}
\begin{equation*}
				W(\underline{u};x_j,\nu_j^{\II})=\inf_{\abs\nu=\varepsilon}W(\underline{u};x_j,\nu)
\end{equation*}
for all $j\geq 0$, where $W$ denotes the auxiliary function defined in \eqref{Wu}. By a measure theoretical analysis, we can prove that this strategy is Borel measurable. For more details, see for example \cite[Theorem 5.3.1]{srivastava98}.

Fix any strategy $\mathcal{S}_\I$ for $P_\I$, and let us define a function $\Phi:\widetilde{C}\times \mathbb R^n\to\mathbb R$ such that
\begin{align*}
\Phi(c,x)=\begin{cases}
\underline{u}(x),&~~\text{if }c=0, \\
F(x),&~~\text{if }c=1.
\end{cases}
\end{align*}
Then, we can estimate
%\begin{align*}
%\begin{split}
%&\E_{\mathcal{S}_\I,\mathcal{S}^*_{\II}}^{x_0}[\underline u(x_{j+1})|x_0,\ldots,x_j] \\
%= ~& \frac{1-\delta(x_j)}{2}\Bigg[\pare{\alpha(x_j)\underline u(x_j+\nu_{j+1}^\I)+\beta(x_j)\dashint_{B_\varepsilon^{\nu_{j+1}^\I}}\underline u(x_j+h)\dL(h)} \\
%~& +\pare{\alpha(x_j)\underline u(x_j+\nu_{j+1}^{\II})+\beta(x_j)\dashint_{B_\varepsilon^{\nu_{j+1}^{\II}}}\underline u(x_j+h)\dL(h)}\Bigg] \\
%~& +\delta(x_j)F(x_j) \\			
%\leq ~& \frac{1-\delta(x_j)}{2}\Bigg[\sup_{\abs\nu=\varepsilon}\pare{\alpha(x_j)\underline u(x_j+\nu)+\beta(x_j)\dashint_{B_\varepsilon^{\nu_j}}\underline u(x_j+h)\dL(h)} \\
%~&+\inf_{\abs\nu=\varepsilon}\pare{\alpha(x_j)\underline u(x_j+\nu)+\beta(x_j)\dashint_{B_\varepsilon^{\nu_j}}\underline u(x_j+h)\dL(h)}\Bigg] \\
%~& +\delta(x_j)F(x_j).
%\end{split}
%\end{align*}
\begin{align*}
\begin{split}
&\E_{\mathcal{S}_\I,\mathcal{S}^*_{\II}}^{x_0}\Big[\Phi (c_{j+1},x_{j+1})\big|\big((c_0,x_0),\ldots,(c_j,x_j)\big)\Big] \\
= ~& \bigg(\frac{1-\delta(x_j)}{2}\brackets{W(\underline{u};x_j,\nu_{j+1}^\I)+W(\underline{u};x_j,\nu_{j+1}^{\II})} +\delta(x_j)F(x_j)\bigg)\mathbb I_{c_j}\big(\{0\}\big) \\	
&\hspace{1em}+F(x_j)\mathbb I_{c_j}\big(\{1\}\big)\\		
\leq ~& \bigg(\frac{1-\delta(x_j)}{2}\brackets{\sup_{\abs\nu=\varepsilon}W(\underline{u};x_j,\nu)+\inf_{\abs\nu=\varepsilon}W(\underline{u};x_j,\nu)} +\delta(x_j)F(x_j)\bigg)\mathbb I_{c_j}\big(\{0\}\big) \\
&\hspace{1em}+F(x_j)\mathbb I_{c_j}\big(\{1\}\big).
\end{split}
\end{align*}
Since $\underline u$ is a solution to \eqref{DPP}, we have by this estimate
\begin{equation*}
				\E_{\mathcal{S}_\I,\mathcal{S}^*_{\II}}^{x_0}\Big[\Phi(c_{j+1},x_{j+1})\big|\big((c_0,x_0),\ldots,(c_j,x_j)\big)\Big] \leq \Phi(c_j,x_j)
\end{equation*}
for all $j\geq 0$. Thus, the stochastic process $\big(\Phi(c_k,x_k)\big)_{k=0}^\infty$ is a supermartingale, when $P_{\II}$ uses the strategy $\mathcal{S}^*_{\II}$. By recalling \eqref{rv tau}, and since $F$ is bounded, we get by using the optional stopping theorem that
\begin{align*}
				u_{\II}(x_0)&=  \inf_{\mathcal{S}_{\II}}\sup_{\mathcal{S}_\I}\E_{\mathcal{S}_\I,\mathcal{S}_{\II}}^{x_0}\big[F(x_\tau)\big]  \\
				&\leq   \sup_{\mathcal{S}_\I}\E_{\mathcal{S}_\I,\mathcal{S}^*_{\II}}^{x_0}\big[F(x_\tau)\big] \\
&= \sup_{\mathcal{S}_\I}\E_{\mathcal{S}_\I,\mathcal{S}^*_{\II}}^{x_0}\big[\Phi(c_{\tau+1},x_{\tau+1})\big] \\
				&\leq  \underline u(x_0),
\end{align*}
because it holds $c_0=0$. Therefore, the proof is complete.
\end{proof}

Now, \Cref{LEMMA-UINEQ} and \Cref{LEMMA-MEASSOL1} imply the uniqueness of solutions to \eqref{DPP}. In addition, this unique function is continuous and the value function of the game.
\begin{theorem}\label{value unique}
Let $\varepsilon>0$ and let $F:\Gamma_{\varepsilon,\varepsilon}\to\mathbb R^n$ be a continuous function. Then, there exists a continuous function $u_\varepsilon: \ol{\Omega}_\varepsilon \to \mathbb R^n$ with the boundary data $F$ such that it satisfies the dynamic programming principle \eqref{DPP}. Moreover, this function is unique and it is the value function of the game, i.e., $u_\varepsilon=u_\I=u_{\II}$ with $u_\I$ and $u_{\II}$ defined in \eqref{uIIIdef}.
\end{theorem}
\section{Local regularity}\label{sec:local regularity}

In this section, we give a local regularity estimate for functions satisfying \eqref{DPP} in $\Omega\setminus I_\varepsilon$. The dynamic programming principle in $\Omega\setminus I_\varepsilon$ reduces to the equation
\begin{equation}\label{local DPP}
\begin{split}
				u(x) = ~& \displaystyle\frac{1}{2}\Bigg[\displaystyle \sup_{\abs\nu=\varepsilon}\pare{\alpha(x)u(x+\nu)+\beta(x)\dashint_{B_\varepsilon^\nu} u(x+h)\dL(h)} \\
				& + \displaystyle \inf_{\abs\nu=\varepsilon}\pare{\alpha(x)u(x+\nu)+\beta(x)\dashint_{B_\varepsilon^\nu} u(x+h)\dL(h)}\Bigg]. 		\end{split}
\end{equation}

The regularity result is based on a method established by Luiro and Parviainen in \cite{luirop}. The method consists of several steps. First, we choose a comparison function $f$ having the desired regularity properties. Then, the idea is to analyze two different cases separately. At a small scale, we need to control the effects arising from the discretization. At a bigger scale, the key term of the comparison function is $C|x-z|^\gamma$ with $x,z\in \mathbb R^n, 0<\gamma<1$ and $C>0$ big enough.

In the second step, we aim to prove that the error $u(x)-u(z)-f(x,z)$, where $u$ is the solution to \eqref{local DPP}, is smaller in $(B_1 \times B_1 )\setminus T$ than in $(B_2 \times B_2) \setminus (B_1 \times B_1 \setminus T)$ with both sets belonging to $\mathbb R^{2n}$. The set $T$ is the set of points $(x,z) \in \mathbb R^{2n}$ such that $x=z$. Then, we thrive for a contradiction by assuming that the error is bigger in $(B_1 \times B_1 )\setminus T$.

As a final step, we get a contradiction by using a multidimensional dynamic programming principle for the comparison function $f$. In the proof below, intuition based on suitable strategies is helpful even though we do not write down stochastic arguments.

\begin{theorem}\label{LOCREG}
Let $(x,z)\in B_R\times B_R$, $B_{2R}\subset\Omega$ and
\begin{equation}\label{gamma exponent}
0<\gamma<\frac{\alpha_{\text{min}}}{\alpha_{\text{max}}}-\kappa
\end{equation}
for arbitrary small $\kappa\in\big(0,\alpha_{\text{min}}/\alpha_{\text{max}}\big)$ with $\alpha_{\text{min}},\alpha_{\text{max}}$ defined in \eqref{alpha assump}. Then, if $u$ satisfies \eqref{local DPP}, we have
\begin{equation}\label{Hoelder}
				\abs{u(x)-u(z)}\leq C\frac{\abs{x-z}^\gamma}{R^\gamma}+C\frac{\varepsilon^\gamma}{R^\gamma}
\end{equation}
with $C:=C(p_{\textrm{min}},p_{\textrm{max}},n,R,\sup_{B_{2R}} u,\gamma)$, $0<\varepsilon<1$ and $p_{\textrm{min}},p_{\textrm{max}}$ defined in \eqref{function p assump}.
\end{theorem}
\begin{proof}
By using a scaling $x \mapsto Rx$, we can assume that $R=1$. In addition by translation, it is enough to consider the claim \eqref{Hoelder} in the case $z=-x$. For simplicity, we assume $\sup_{B_{2}\times B_{2}}(u(x)-u(z))\leq 1$.

Given $C>1$, let $N\in\N$ be such that
\begin{equation*}
				N\geq\frac{10^2C}{\gamma}.
\end{equation*}
Then, we define the following functions in $\R^{2n}$
\begin{align*}
f_1(x,z)& = C\abs{x-z}^\gamma+\abs{x+z}^2, \\
f_2(x,z)& = \left\{\begin{array}{ll}
								C^{2(N-i)}\varepsilon^\gamma & \mbox{ if } (x,z)\in A_i, \\
								0 & \mbox{ if } \abs{x-z}>N\displaystyle\frac{\varepsilon}{10},
				\end{array}\right.\\
f(x,z) &= f_1(x,z)-f_2(x,z)
\end{align*}
with $\displaystyle A_i=\set{(x,z)\in\R^{2n}}{(i-1)\frac{\varepsilon}{10}<\abs{x-z}\leq i\frac{\varepsilon}{10}}$ for $i=0,1,\ldots,N$. The function $f_2$ is called an \textit{annular step function}, and it is needed to control the small scale jumps. Note that we have $\sup f_2=C^{2N}\varepsilon^\gamma$ reached on
\begin{equation*}
				T:=A_0=\set{(x,z)\in\R^{2n}}{x=z}.
\end{equation*}

It holds that $f_1\geq 1$ in $(B_2\times B_2)\setminus(B_1\times B_1)$. Here, we need the term $|x+z|^2$ in the function $f_1$, because
$$
|x+z|^2=2|x|^2+2|z|^2-|x-z|^2\geq 3
$$
for all $x,z\in (B_2\times B_2)\setminus(B_1\times B_1)$ such that $|x-z|\leq 1$. Therefore, together with $u(x)-u(z)\leq 1$ in $B_2\times B_2$ and $u(x)-u(z)=0$ in $T$ we have
\begin{equation}\label{MEq}
				u(x)-u(z)-f(x,z)\leq \sup f_2=C^{2N}\varepsilon^\gamma,
\end{equation}
if $(x,z)\in T$ or $(x,z) \in (B_2\times B_2)\setminus(B_1\times B_1)$. We have to show that this inequality is also true in $(B_1\times B_1)\setminus T$. Thriving for a contradiction, write
\begin{equation*}
				M:=\sup_{(x,z)\in B_1\times B_1\setminus T}(u(x)-u(z)-f(x,z))
\end{equation*}
and suppose that
\begin{equation*}
				M>C^{2N}\varepsilon^\gamma.
\end{equation*}
By \eqref{MEq}, this is equivalent to
\begin{equation}\label{RaA}
				M=\sup_{(x,z)\in B_2\times B_2}(u(x)-u(z)-f(x,z)).
\end{equation}
For all $\eta>0$, we choose a pair of points $(x,z)\in( B_1\times B_1)\setminus T$ such that
\begin{equation}\label{INEQ01}
				M\leq u(x)-u(z)-f(x,z)+\frac{\eta}{2}.
\end{equation}
Then by \eqref{local DPP}, we have
\begin{equation}\label{ineqdpp}
\begin{split}
			u(x)-u(z) \leq ~& \frac{1}{2}\sup_{\nu_x,\nu_z}(W(x,\nu_x)-W(z,\nu_z)) +\frac{1}{2}\inf_{\nu_x,\nu_z}(W(x,\nu_x)-W(z,\nu_z)),
\end{split}
\end{equation}
where $W$ is the auxiliary function defined in \eqref{Wu}.

Given $\abs{\nu_x}=\abs{\nu_z}=\varepsilon$, let $P_{\nu_z,-\nu_x}$ denote any rotation that sends $\nu_z$ to $-\nu_x$. By recalling $\alpha(x)+\beta(x)=1$ for $x \in \Omega$, we can decompose the difference $W(x,\nu_x)-W(z,\nu_z)$. For simplicity, we may assume that $\alpha(x)\geq\alpha(z)$. Thus, we get
\begin{equation}\label{eq001}
\begin{split}
		  W(x,\nu_x)&-W(z,\nu_z)
= \alpha(z)\brackets{u(x+\nu_x) -  u(z+\nu_z)} \\
			~&\hspace{12pt} +  \beta(x)\displaystyle\dashint_{B_\varepsilon^{\nu_z}}\brackets{u(x+P_{\nu_z,-\nu_x}h) -  u(z+h)}\dL(h) \\
			~& \hspace{12pt} +  (\alpha(x)-\alpha(z))\brackets{u(x+\nu_x)-\displaystyle\dashint_{B_\varepsilon^{\nu_z}} u(z+h)\dL(h)}.
\end{split}
\end{equation}
Next, we use the counter assumption \eqref{RaA} to estimate each of the terms in \eqref{eq001} from above. Consequently, we can estimate
$$
u(y)-u(\tilde{y})\leq M+f(y,\tilde{y})
$$
for all $y,\tilde{y}\in B_2$. Then, we define
\begin{equation}\label{Qdef}
\begin{split}
			G(f,x,z,\nu_x,\nu_z) := ~& \alpha(z)f(x+\nu_x,z+\nu_z) \\
			~& +  \beta(x)\displaystyle\dashint_{B_\varepsilon^{\nu_z}} f(x+P_{\nu_z,-\nu_x}h,z+h)\dL(h) \\
			~&+  (\alpha(x)-\alpha(z))\displaystyle\dashint_{B_\varepsilon^{\nu_z}} f(x+\nu_x,z+h)\dL(h).
\end{split}
\end{equation}
Thus, we have
\begin{equation}\label{wmq}
				W(x,\nu_x)-W(z,\nu_z) \leq M + G(f,x,z,\nu_x,\nu_z).
\end{equation}
By taking the supremum, we obtain
\begin{equation}\label{ineqwmq}
			\sup_{\nu_x,\nu_z}(W(x,\nu_x)-W(z,\nu_z)) \leq M +\sup_{\nu_x,\nu_z}G(f,x,z,\nu_x,\nu_z).
\end{equation}
On the other hand, choose $\abs{\varrho_x}=\abs{\varrho_z}=\varepsilon$ such that
\begin{equation*}
			\inf_{\nu_x,\nu_z}G(f,x,z,\nu_x,\nu_z) \geq G(f,x,z,\varrho_x,\varrho_z)-\eta.
\end{equation*}
This together with \eqref{wmq} yields
\begin{align*}
			\inf_{\nu_x,\nu_z}(W(x,\nu_x)-W(z,\nu_z))  \leq & W(x,\varrho_x)-W(z,\varrho_z) \\
			 \leq & M + G(f,x,z,\varrho_x,\varrho_z) \\
			 \leq & M+\inf_{\nu_x,\nu_z}G(f,x,z,\nu_x,\nu_z)+\eta.
\end{align*}
Therefore, by applying this inequality and \eqref{ineqwmq} to \eqref{ineqdpp} we get
\begin{equation*}
			u(x)-u(z) \leq M+\frac{1}{2}\brackets{\sup_{\nu_x,\nu_z}G(f,x,z,\nu_x,\nu_z)+\inf_{\nu_x,\nu_z}G(f,x,z,\nu_x,\nu_z)}+\frac{\eta}{2}.
\end{equation*}
Combining this with \eqref{INEQ01}, we need to show
$$
\sup_{\nu_x,\nu_z}G(f,x,z,\nu_x,\nu_z)+\inf_{\nu_x,\nu_z}G(f,x,z,\nu_x,\nu_z)<2f(x,z).
$$
This inequality follows from \Cref{P1} below. Consequently, the equation \eqref{MEq} holds in $B_2\times B_2$.
\end{proof}

\begin{proposition}\label{P1}
Let $f$ and $T$ be as at the beginning of the proof of \Cref{LOCREG}, and fix $x,z \in B_1\times B_1 \setminus T$. In addition, let $G$ be as in \eqref{Qdef}. Then, it holds that
\begin{equation*}
				\sup_{\nu_x,\nu_z}G(f,x,z,\nu_x,\nu_z)+\inf_{\nu_x,\nu_z}G(f,x,z,\nu_x,\nu_z)<2f(x,z).
\end{equation*}
\end{proposition}

The main part of the section is to show this estimate for $G$. This is done in several steps below.

\subsection{Proof of \Cref{P1}}

Let $V\subset\R^n$ be the space spanned by $x-z\neq 0$. We denote the orthogonal complement of $V$ by $V^\bot$, i.e.,
\begin{align*}
				V^\bot=\set{y\in\R^n}{\langle y,x-z \rangle =0}.
\end{align*}
Given any $y\in\R^n$, we can decompose $y=y_V(x-z)/|x-z|+y_{V^\bot}$, where $y_V\in\R$ is the scalar projection of $y$ onto $V$ and $y_{V^\bot}\in V^\bot$, respectively. For the decomposed point it holds
\begin{align*}
				y_V&=\prodin{y}{\frac{x-z}{\abs{x-z}}},\\
								\abs{y_{V^\bot}}&=\sqrt{\abs{y}^2-y_V^2}.
\end{align*}
By using this notation, the second order Taylor's expansion of $f_1$ is
\begin{equation}\label{fTaylor}
\begin{split}
				~& \hspace{-33pt} f_1(x+h_x,z+h_z) - f_1(x,z) \\
				= ~& C\gamma\abs{x-z}^{\gamma-1}(h_x-h_z)_V+2\prodin{x+z}{h_x+h_z} \\
				& + \frac{1}{2}C\gamma\abs{x-z}^{\gamma-2}\braces{(\gamma-1)(h_x-h_z)_V^2+\abs{(h_x-h_z)_{V^\bot}}^2} \\
				& + \abs{h_x+h_z}^2 + \mathcal{E}_{x,z}(h_x,h_z),
\end{split}
\end{equation}
where $\mathcal{E}_{x,z}(h_x,h_z)$ is the error term. In the above, we used the calculations
$$
\prodin{\nabla f_1(x,z)}{(h_x^T,h_z^T)}=C\gamma|x-z|^{\gamma-2}\prodin{x-z}{h_x-h_z}+2\prodin{x+z}{h_x+h_z}
$$
and
$$
D^2f_1=\left[ \begin{array}{cc}A & -A\\
-A & A \end{array} \right] + 2\left[ \begin{array}{cc}\text{I} & \text{I} \\
\text{I}  & \text{I}\end{array} \right]
$$
with
$$
A :=C\gamma |x-z|^{\gamma-2} \bigg[(\gamma -2)\frac{x-z}{|x-z|}\otimes \frac{x-z}{|x-z|}+\textrm{I} \bigg].
$$
The matrix $\text{I}$ stands for the $n\times n$ identity matrix, and we denote the tensor product of two vectors by $\otimes$, i.e., $h\otimes s:=hs^T$ for vectors $h,s \in \mathbb R^n$. By recalling the elementary formula $h^T (s\otimes s) h=\prodin{h}{s}^2$ for all $h,s\in \mathbb R^n$, we get \eqref{fTaylor}.

By Taylor's theorem, the error term satisfies
\begin{equation*}
				\abs{\mathcal{E}_{x,z}(h_x,h_z)}\leq C\abs{(h_x^T,h_z^T)}^3(\abs{x-z}-2\varepsilon)^{\gamma-3},
\end{equation*}
if $\abs{x-z}>2\varepsilon$. With the choice $N\geq\frac{100C}{\gamma}$ and if $\abs{x-z}>\frac{N}{10}\varepsilon$, we can estimate
\begin{align}
\begin{split}\label{ineqerror}
				\abs{\mathcal{E}_{x,z}(h_x,h_z)}&\leq C(2\varepsilon)^3\bigg(\frac{|x-z|}{2}\bigg)^{\gamma-3} \\
&\leq 64C\varepsilon^2|x-z|^{\gamma-2}\frac{\varepsilon}{|x-z|}\\
&\leq 10\abs{x-z}^{\gamma-2}\varepsilon^2,
\end{split}
\end{align}
because $\abs{h_x},\abs{h_z}\leq\varepsilon$. Therefore, to prove the result, we distinguish two separate cases. In the first case, we have $\abs{x-z}\leq\frac{N}{10}\varepsilon$ and in the second case, we have $\abs{x-z}>\frac{N}{10}\varepsilon$.
\subsubsection*{Proof of \Cref{P1}: Case $\abs{x-z}\leq N\frac{\varepsilon}{10}$}

In this case, we do not utilize the formula \eqref{fTaylor}. We use concavity and convexity estimates for the terms in $f_1$ and the properties of the annular step function $f_2$. For $x,z\in B_1$ and $\abs{h_x},\abs{h_z}<\varepsilon<1$, it holds
\begin{equation*}
				\abs{f_1(x+h_x,z+h_z)-f_1(x,z)}\leq 2C\varepsilon^\gamma+16\varepsilon\leq3C\varepsilon^\gamma
\end{equation*}
for $C>16$. Consequently by \eqref{Qdef}, we have
\begin{equation*}
				\sup_{h_x,h_z}G(f_1,x,z,h_x,h_z)\leq f_1(x,z)+3C\varepsilon^\gamma.
\end{equation*}
Together with $f_2\geq 0$, these estimates yield
\begin{equation}\label{eq101}
				\sup_{h_x,h_z} G(f,x,z,h_x,h_z)\leq f_1(x,z)+3C\varepsilon^\gamma.
\end{equation}
Find $i\in \{1,2,\ldots,N\}$ such that $(i-1)\frac{\varepsilon}{10}<\abs{x-z}\leq i\frac{\varepsilon}{10}$ and choose $\abs{\nu_x},\abs{\nu_z}<\varepsilon$ such that $(x+\nu_x,z+\nu_z)\in A_{i-1}$. Then for $C>1$ large enough, we can estimate
\begin{equation*}
\begin{split}
				\sup_{h_x,h_z}G(f_2,x,z,h_x,h_z) \geq ~& G(f_2,x,z,\nu_x,\nu_z) \\
				\geq ~& \alpha(z)f_2(x+\nu_x,z+\nu_z) \\
				= ~& \alpha(z)C^{2(N-i+1)}\varepsilon^\gamma \\
				= ~& \alpha(z)\pare{C^2-\frac{2}{\alpha(z)}}C^{2(N-i)}\varepsilon^\gamma+2f_2(x,z) \\
				> ~& 6C\varepsilon^\gamma+2f_2(x,z),
\end{split}
\end{equation*}
where we use $f_2\geq 0$ in the second inequality and $\alpha(z)>\alpha_{\text{min}}>0$ for all $z\in \Omega$ in the last inequality. Therefore, by $f=f_1-f_2$ and \eqref{eq101} it holds
\begin{align*}
\begin{split}
				\inf_{h_x,h_z}G(f,x,z,h_x,h_z) \leq  ~& \sup_{h_x,h_z} G(f_1,x,z,h_x,h_z)-\sup_{h_x,h_z} G(f_2,x,z,h_x,h_z))\\
				\leq ~& f_1(x,z)-2f_2(x,z)-3C\varepsilon^\gamma.
\end{split}
\end{align*}
Combining this inequality with \eqref{eq101}, we get
\begin{equation*}
				\sup_{h_x,h_z}G(f,x,z,h_x,h_z)+\inf_{h_x,h_z}G(f,x,z,h_x,h_z) < 2f(x,z).
\end{equation*}
Hence, the proof of the case is complete.

\subsubsection*{Proof of \Cref{P1}: Case $\abs{x-z}>N\frac{\varepsilon}{10}$}

In this case, $f_2(x,z)=0$ and hence $f\equiv f_1$. We apply \eqref{fTaylor} to get the result. For $\eta>0$, let $\nu_x$, $\nu_z$ be such that
\begin{equation*}
				\sup_{h_x,h_z}G(f,x,z,h_x,h_z)\leq G(f,x,z,\nu_x,\nu_z)+\eta.
\end{equation*}
Therefore for any $\abs{\varrho_x},\abs{\varrho_z}\leq\varepsilon$, we get the following inequality
\begin{align}
\begin{split}\label{ineqEXTRA}
				&\sup_{h_x,h_z}G(f,x,z,h_x,h_z) +\inf_{h_x,h_z}G(f,x,z,h_x,h_z) \\
&\leq G(f,x,z,\nu_x,\nu_z)+G(f,x,z,\varrho_x,\varrho_z)+\eta.
\end{split}
\end{align}
By \eqref{ineqerror} and $\abs{h_x},\abs{h_z}\leq\varepsilon$, the last two terms in \eqref{fTaylor} are bounded above by
\begin{equation*}
				(4+10\abs{x-z}^{\gamma-2})\varepsilon^2.
\end{equation*}
We denote
\begin{equation*}
E:=E(f,x,z,\gamma,\varepsilon):=f(x,z)+(4+10\abs{x-z}^{\gamma-2})\varepsilon^2,
\end{equation*}
and recall the notation $P_{h,s}$ denoting the rotation sending $h$ to $s$ for any vectors $\abs{h}=\abs{s}$ in $\mathbb R^n$. By \eqref{ineqEXTRA} and \eqref{Qdef}, it suffices to study
\begin{equation}\label{111}
\begin{split}
				\textbf{[I]} := ~& G(f,x,z,\nu_x,\nu_z)+G(f,x,z,\varrho_x,\varrho_z)-2E\\
				= ~& \alpha(z)\brackets{f(x+\nu_x,z+\nu_z)+f(x+\varrho_x,z+\varrho_z)-2E} \\
				~& + \beta(x)\Bigg[\displaystyle\dashint_{B_\varepsilon^{\nu_z}}  f(x+P_{\nu_z,-\nu_x}h,z+h)\dL(h) \\
				~&  +\displaystyle\dashint_{B_\varepsilon^{\varrho_z}} f(x+P_{\varrho_z,-\varrho_x}h,z+h)\dL(h)-2E\Bigg] \\
				~& + (\alpha(x)-\alpha(z))\Bigg[\displaystyle\dashint_{B_\varepsilon^{\nu_z}} f(x+\nu_x,z+h)\dL(h) \\
				~& +\displaystyle\dashint_{B_\varepsilon^{\varrho_z}} f(x+\varrho_x,z+h)\dL(h)-2E\Bigg].				
\end{split}
\end{equation}
For simplicity, we decompose the previous expression into three terms to be examined separately, i.e.,
\begin{equation}\label{234}
				 \textbf{[I]} = \alpha(z)\textbf{[II]}+\beta(x)\textbf{[III]}+(\alpha(x)-\alpha(z))\textbf{[IV]}.
\end{equation}
Then by \eqref{fTaylor}, we have
\begin{equation}\label{II}
\begin{split}
			\textbf{[II]} \leq ~& C\gamma\abs{x-z}^{\gamma-1}\brackets{(\nu_x-\nu_z)_V+(\varrho_x-\varrho_z)_V} \\
			~& + 2\prodin{x+z}{(\nu_x+\nu_z)+(\varrho_x+\varrho_z)} \\
			~& + \frac{1}{2}C\gamma\abs{x-z}^{\gamma-2}\Big\{(\gamma-1)\brackets{(\nu_x-\nu_z)^2_V+(\varrho_x-\varrho_z)^2_V} \\
			~& +\brackets{\abs{(\nu_x-\nu_z)_{V^\bot}}^2+\abs{(\varrho_x-\varrho_z)_{V^\bot}}^2}\Big\}.
\end{split}
\end{equation}
Note that the first order terms in \textbf{[III]} vanishes when we integrate over the ball. Therefore, we can estimate
\begin{equation}\label{III}
\begin{split}
			\textbf{[III]} \leq ~& \frac{1}{2}C\gamma\abs{x-z}^{\gamma-2}\cdot \\
			~& \cdot\Bigg\{\displaystyle\dashint_{B_\varepsilon^{\nu_z}}\brackets{(\gamma-1) (h-P_{\nu_z,-\nu_x}h)^2_V+\abs{(h-P_{\nu_z,-\nu_x}h)_{V^\bot}}^2}\dL(h) \\
			~&  +\displaystyle\dashint_{B_\varepsilon^{\varrho_z}}\brackets{(\gamma-1)(h-P_{\varrho_z,-\varrho_x}h)^2_V+\abs{(h-P_{\varrho_z,-\varrho_x}h)_{V^\bot}}^2}\dL(h)\Bigg\}.
\end{split}
\end{equation}
In addition, it holds
\begin{equation}\label{IV}
\begin{split}
			\textbf{[IV]} \leq ~& C\gamma\abs{x-z}^{\gamma-1}(\nu_x+\varrho_x)_V + 2\prodin{x+z}{\nu_x+\varrho_x} \\
			~& + \frac{1}{2}C\gamma\abs{x-z}^{\gamma-2}\cdot \\
			~& \cdot\Bigg\{\displaystyle\dashint_{B_\varepsilon^{\nu_z}}\brackets{(\gamma-1)(\nu_x-h)^2_V+\abs{(\nu_x-h)_{V^\bot}}^2}\dL(h) \\
			~& \hspace{11pt}  +\displaystyle\dashint_{B_\varepsilon^{\varrho_z}}\brackets{(\gamma-1)(\varrho_x-h)^2_V+\abs{(\varrho_x-h)_{V^\bot}}^2}\dL(h)\Bigg\}.
\end{split}
\end{equation}

We distinguish between two cases depending on the value of $(\nu_x-\nu_z)^2_V$ and fix $\tau_0<\tau<1$ with $0<\tau_0<1$ defined later.

\paragraph{a) Case $\abs{(\nu_x-\nu_z)_V}\geq(\tau+1)\varepsilon$:}
In this case, we choose $\varrho_x=-\nu_x$ and $\varrho_z=-\nu_z$. By replacing these vectors in the inequalities $\textbf{[II]}$, $\textbf{[III]}$ and $\textbf{[IV]}$ and using symmetry, we obtain
\begin{equation*}
\begin{split}
			\textbf{[II]} \leq ~& C\gamma\abs{x-z}^{\gamma-2}\brackets{(\gamma-1)(\nu_x-\nu_z)^2_V+\abs{(\nu_x-\nu_z)_{V^\bot}}^2}, \\
			\textbf{[III]} \leq ~& C\gamma\abs{x-z}^{\gamma-2}\displaystyle\dashint_{B_\varepsilon^{\nu_z}}\abs{(h-P_{\nu_z,-\nu_x}h)_{V^\bot}}^2\dL(h), \\
			\textbf{[IV]} \leq ~& C\gamma\abs{x\!-\!z}^{\gamma-2}\Bigg[(\gamma\!-\!1)\displaystyle\dashint_{B_\varepsilon^{\nu_z}}(\nu_x\!-\!h)^2_V\dL(h) \!+\!\displaystyle\dashint_{B_\varepsilon^{\nu_z}}\abs{(\nu_x\!-\!h)_{V^\bot}}^2\dL(h)\Bigg].
\end{split}
\end{equation*}
We used $\gamma-1<0$ and the choice $P_{\nu_z,-\nu_x}=P_{-\nu_z,\nu_x}$ in the estimate for $\textbf{[III]}$. By assumption, it holds $(\nu_x-\nu_z)^2_V\geq(\tau+1)^2\varepsilon^2$ implying
$$
\abs{(\nu_x-\nu_z)_{V^\bot}}^2\leq\brackets{4-(\tau+1)^2}\varepsilon^2.
$$
Thus, we need to obtain uniform bounds for the terms $(\nu_x-h)^2_V$, $\abs{(\nu_x-h)_{V^\bot}}^2$ and $\abs{(h-P_{\nu_z,-\nu_x}h)_{V^\bot}}^2$ for $h\in B_\varepsilon^{\nu_z}$.

The assumption $\abs{(\nu_x-\nu_z)_V}\geq(\tau+1)\varepsilon$, together with $|\nu_x|,|\nu_z|\leq \varepsilon$ and Pythagoras' theorem, implies
\begin{align}\begin{split}\label{est on vectors}
				\left\{\begin{array}{lll}
								\tau\varepsilon\leq\abs{(\nu_x)_V}\leq\varepsilon, & & 0\leq\abs{(\nu_x)_{V^\bot}}\leq\sqrt{1-\tau^2}\ \varepsilon, \\ 
								\tau\varepsilon\leq\abs{(\nu_z)_V}\leq\varepsilon, & & 0\leq\abs{(\nu_z)_{V^\bot}}\leq\sqrt{1-\tau^2}\ \varepsilon.
				\end{array}\right.
\end{split}
\end{align}
Moreover, the same facts yield
\begin{equation*}
				\abs{(\nu_x+\nu_z)_V}\leq (1-\tau)\varepsilon \ \ \ \mbox{ and } \ \ \ \abs{(\nu_x+\nu_z)_{V^\bot}}\leq 2\sqrt{1-\tau^2}\ \varepsilon.
\end{equation*}
By combining these and using Pythagoras' theorem, we get
\begin{equation}\label{EQ-A1}
				\abs{\nu_x+\nu_z}<\sqrt{8}\ \sqrt{1-\tau}\ \varepsilon,
\end{equation}
since $\tau<1$. Let $h\in B_\varepsilon^{\nu_z}$. Then, we have
$$
0=\prodin{h}{\nu_z}=h_V(\nu_z)_V+\prodin{h_{V^\bot}}{(\nu_z)_{V^\bot}}
$$
implying
\begin{equation*}
				h_V=-\frac{\prodin{h_{V^\bot}}{(\nu_z)_{V^\bot}}}{(\nu_z)_V}.
\end{equation*}
In addition by applying this equality together with \eqref{est on vectors} and $\abs{h_{V^\bot}}\leq\abs{h}\leq\varepsilon$, we obtain
\begin{equation*}
\abs{h_V}\leq\displaystyle\frac{\varepsilon}{\tau}\sqrt{1-\tau^2}.								
\end{equation*}
Consequently, we get the estimates
\begin{equation}\label{estimate01}
(\nu_x-h)_V \geq \bigg(\tau-\frac{\sqrt{1-\tau^2}}{\tau}\bigg)\varepsilon
\end{equation}
and
\begin{equation}\label{estimate02}
\abs{(\nu_x-h)_{V^\bot}} \leq \abs{(\nu_x)_{V^\bot}}+\abs{h_{V^\bot}} \leq \pare{1+\sqrt{1-\tau^2}}\varepsilon.
\end{equation}
We can assume that $\tau_0$ is close enough to 1 guaranteeing the positivity of the quantity $\tau-\tau^{-1}\sqrt{1-\tau^2}$. In order to obtain the last estimate needed, we recall that $P_{\nu_z,-\nu_x}$ is any rotation sending the vector $\nu_z$ to $-\nu_x$. In particular, we choose a rotation satisfying
\begin{equation*}
				\abs{h-P_{\nu_z,-\nu_x}h}\leq\abs{\nu_z-P_{\nu_z,-\nu_x}\nu_z}=|\nu_z+\nu_x|
\end{equation*}
for every $\abs{h}\leq\varepsilon$. Hence by recalling \eqref{EQ-A1}, we get
\begin{equation}\label{estimate03}
				\abs{(h-P_{\nu_z,-\nu_x}h)_{V^\bot}}^2\leq 8(1-\tau)\varepsilon^2.
\end{equation}

By replacing the estimates \eqref{estimate01}, \eqref{estimate02} and \eqref{estimate03} in $\textbf{[II]}$, $\textbf{[III]}$ and $\textbf{[IV]}$, we can calculate
\begin{equation*}
\begin{split}
			\textbf{[II]} \leq ~& C\gamma\abs{x-z}^{\gamma-2}\varepsilon^2\brackets{(\gamma-1)(\tau+1)^2+4-(\tau+1)^2}, \\
			 \textbf{[III]} \leq ~& C\gamma\abs{x-z}^{\gamma-2}\varepsilon^2\brackets{8(1-\tau)}, \\
			 	\textbf{[IV]} \leq ~& C\gamma\abs{x-z}^{\gamma-2}\varepsilon^2\brackets{(\gamma-1)\pare{\tau-\frac{\sqrt{1-\tau^2}}{\tau}}^2+\pare{1+\sqrt{1-\tau^2}}^2}.
\end{split}
\end{equation*}
In addition by \eqref{234}, we get
\begin{equation*}
					\textbf{[I]} \leq \textbf{[V]}\cdot C\gamma\abs{x-z}^{\gamma-2}\varepsilon^2,
\end{equation*}
where $\textbf{[V]}$ is equal to
\begin{multline*}
				(\gamma-1)\brackets{\alpha(z)(\tau+1)^2+(\alpha(x)-\alpha(z))\pare{\tau-\frac{\sqrt{1-\tau^2}}{\tau}}^2} \\
				+ (\alpha(x)-\alpha(z))(1+\sqrt{1-\tau^2}\ )^2+
				\alpha(z)\brackets{(4-(\tau+1)^2)}+\beta(x)\ 8(1-\tau).
\end{multline*}
The assumption on $\gamma$ in \eqref{gamma exponent} implies that we can choose $\tau_0:=\tau_0(\kappa)<1$ close enough to $1$ such that the previous expression is negative, i.e.,
 \begin{align*}
\textbf{[V]}&<(\gamma-1)\big(4\alpha(z)+\alpha(x)-\alpha(z)\big)+\alpha(x)-\alpha(z)+\kappa\alpha_{\text{max}} \\
&<4\big(\gamma\alpha_{\text{max}}-\alpha_{\text{min}}\big)+\kappa\alpha_{\text{max}} \\
&<0.
\end{align*}
Now, by recalling \eqref{111}, we have
\begin{multline*}
				G(f,x,z,\nu_x,\nu_z)+G(f,x,z,\omega_x,\omega_z)-2f(x,z)\\
				\leq 8\varepsilon^2+(20+\textbf{[V]}\cdot C\gamma)\abs{x-z}^{\gamma-2}\varepsilon^2.
\end{multline*}
By choosing $C>1$ large enough, we obtain
\begin{equation*}
				(20+\textbf{[V]}\cdot C\gamma)\abs{x-z}^{\gamma-2}\varepsilon^2<-10^8\abs{x-z}^{\gamma-2}\varepsilon^2<-10^7\varepsilon^2.
\end{equation*}
This estimate yields
\begin{equation*}
				G(f,x,z,\nu_x,\nu_z)+G(f,x,z,\omega_x,\omega_z)-2f(x,z)<0.
\end{equation*}

\paragraph{b) Case $\abs{(\nu_x-\nu_z)_V}\leq(\tau+1)\varepsilon$:}
In this case, the first order terms in \eqref{fTaylor} imply the result. By choosing $\varrho_x=-\varepsilon\frac{x-z}{\abs{x-z}}$ and $\varrho_z=\varepsilon\frac{x-z}{\abs{x-z}}$ in $V$ and utilizing these in \eqref{II}, \eqref{III} and \eqref{IV}, we get
\begin{align*}
			\textbf{[II]} \leq ~& C\gamma\abs{x-z}^{\gamma-1}\brackets{(\nu_x-\nu_z)_V-2\varepsilon} +2\prodin{x+z}{\nu_x+\nu_z} \\
			~& +\frac{1}{2}C\gamma\abs{x-z}^{\gamma-2}\braces{(\gamma-1)\brackets{(\nu_x-\nu_z)^2_V+4\varepsilon^2}+\abs{(\nu_x-\nu_z)_{V^\bot}}^2},
\end{align*}
\begin{align*}
			\textbf{[III]} \leq ~& \frac{1}{2}C\gamma\abs{x-z}^{\gamma-2}\cdot \\
			~&\cdot\Bigg\{\displaystyle\dashint_{B_\varepsilon^{\nu_z}}\brackets{(\gamma-1)(h-P_{\nu_z,-\nu_x}h)^2_V+\abs{(h-P_{\nu_z,-\nu_x}h)_{V^\bot}}^2}\dL(h)\Bigg\}
\end{align*}
and
\begin{align*}
			\textbf{[IV]} \leq ~& C\gamma\abs{x-z}^{\gamma-1}\brackets{(\nu_x)_V-\varepsilon} + 2\prodin{x+z}{\nu_x-\varepsilon\frac{x-z}{\abs{x-z}}} \\
			~& + \frac{1}{2}C\gamma\abs{x-z}^{\gamma-2}  \cdot\Bigg\{\displaystyle\dashint_{B_\varepsilon^{\nu_z}}\brackets{(\gamma-1)(\nu_x-h)^2_V+\abs{(\nu_x-h)_{V^\bot}}^2}\dL(h) \\
			~&   +(\gamma-1)\varepsilon^2+\displaystyle\dashint_{B_\varepsilon^{x-z}}\abs{h}^2\dL(h)\Bigg\}.
\end{align*}
The second order terms in these inequalities can be estimated above by
\begin{equation*}
				3C\gamma\abs{x-z}^{\gamma-2}\varepsilon^2.
\end{equation*}
In addition, we deduce that $(\nu_x-\nu_z)_V<\brackets{1+\pare{\frac{\tau+1}{2}}^2}\varepsilon$. Therefore, we have
\begin{equation*}
\begin{split}
			\textbf{[II]} \leq ~& C\gamma\abs{x-z}^{\gamma-1}\brackets{\pare{\frac{\tau+1}{2}}^2-1}\varepsilon +4\abs{x+z}\varepsilon + 3C\gamma\abs{x-z}^{\gamma-2}\varepsilon^2, \\
			 \textbf{[III]} \leq ~& 3C\gamma\abs{x-z}^{\gamma-2}\varepsilon^2, \\
			 \textbf{[IV]} \leq ~& 4\abs{x+z}\varepsilon + 3C\gamma\abs{x-z}^{\gamma-2}\varepsilon^2.
\end{split}
\end{equation*}
By combining all these and recalling \eqref{111} and \eqref{234}, we get
\begin{equation*}
\begin{split}
				& G(f,x,z,\nu_x,\nu_z)+G(f,x,z,\varrho_x,\varrho_z)-2f(x,z) \\
				\leq ~& C\alpha(z)\gamma\abs{x-z}^{\gamma-1}\brackets{\pare{\frac{\tau+1}{2}}^2-1}\varepsilon+4\alpha(x)\abs{x+z}\varepsilon +8\varepsilon^2 \\
				~& +(20+3C\gamma)\abs{x-z}^{\gamma-2}\varepsilon^2 \\
				\leq ~& C\alpha(z)\gamma\abs{x-z}^{\gamma-1}\brackets{\pare{\frac{\tau+1}{2}}^2-1}\varepsilon+8\varepsilon^2+\pare{\gamma+\frac{2}{C}}\gamma\abs{x-z}^{\gamma-1}\varepsilon \\
				\leq  ~& \gamma\abs{x-z}^{\gamma-1}\varepsilon\braces{\gamma+1+C\alpha_{\min}\brackets{\pare{\frac{\tau+1}{2}}^2-1}}+8\varepsilon^2.
\end{split}
\end{equation*}
As in the previous case, we can choose the constant $C>1$ large enough to ensure the negativity of the previous equation. Thus, the proof is complete.

\section{Regularity near the boundary}\label{sec:boundary reg}
In this section, we show that the value function of the game is also asymptotically continuous near the boundary, if we assume some regularity on the boundary of the set. The proof is based on finding a suitable barrier function and a strategy for the other player so that the process under the barrier function is super- or submartingale depending on the form of the function. Then, the result follows by analyzing the barrier function and iterating the argument.

Fix $r>0$ and $z\in \mathbb R^n$ and define a barrier function
\begin{equation}\label{def barrier v}
v(x)=a|x-z|^\sigma+b
\end{equation}
for all $x\in \mathbb R^n\setminus \overline{B}_{r}(z)$ with some constants $\sigma<0$, $a<0$ and $b\geq 0$. Recall the auxiliary function $W$ defined in \eqref{Wu}. First, we prove the following properties of the function $v$.

\begin{lemma}\label{W lemma}
Let $r>0$ and $z\in \mathbb R^n$ and define the function $v$ as in \eqref{def barrier v} with constants $a<0$, $b\geq 0$ and $\sigma<0$. Then, there is a constant $C>0$ such that
\begin{align}\label{W lemma1}
\sup_{|\nu|=\varepsilon} W(v;x,\nu)\leq W\bigg(v;x,\varepsilon\frac{x-z}{|x-z|}\bigg)+ C\varepsilon^3
\end{align}
and
\begin{align}
\begin{split}\label{W lemma2}
&W\bigg(v;x,\varepsilon\frac{x-z}{|x-z|}\bigg)+W\bigg(v;x,-\varepsilon\frac{x-z}{|x-z|}\bigg) \\
&<2 v(x)+\varepsilon^2 a\sigma|x-z|^{\sigma-2}\Big(\beta(x)+\alpha(x)(\sigma -1)\Big) + C\varepsilon^3
\end{split}
\end{align}
for all $\varepsilon>0 $ and $x\in \mathbb R^n\setminus \overline{B}_{r}(z)$.
\begin{proof}
To establish the result, we apply Taylor's formula to the function $v$. The function $v$ is real-analytic so we obtain by Taylor's formula
\begin{align*}
v(x+h) &=v(x)+ \big\langle \nabla v(x),h \big\rangle+\frac{1}{2} \big \langle D^2v(x)h,h\big \rangle+ \mathcal O(|h|^3) \\
&=v(x)\!+\!\frac{1}{2} a\sigma|x\!-\!z|^{\sigma-2}\bigg(2\langle x\!-\!z,h\rangle\!+\!|h|^2 
 \! +\!(\sigma\!-\!2)\frac{\langle h,x\!-\!z\rangle^2}{|x\!-\!z|^2}\bigg) \!+\!\mathcal O(|h|^3)
\end{align*}
for all $h\in \mathbb R^n$. Let $C>0$ be big enough so that $\big|\mathcal O(|h|^3)\big|\leq C\varepsilon^3$ for all $|h|\leq \varepsilon$. The function $v$ is radially increasing, and the average integral over the first order term of $v$ vanishes. In addition, we have
$$
\int_{B_\varepsilon^{x-z}}\langle h,x-z\rangle^2 \-d\mathcal L^{n-1}(h)=0.
$$
Thus, Taylor's formula proves the equation \eqref{W lemma1}.

Next, we prove the equation \eqref{W lemma2}. Recall the notations $V=\text{span} \{x-z\}$ and the orthogonal complement $V^\bot$. For a vector $h\in V$ such that $h=(x-z)\varepsilon/|x-z|$, we get
\begin{align*}
v(x+h)&\leq v(x)+\frac{1}{2} a\sigma|x-z|^{\sigma-2}\Big(2|x-z|\varepsilon+\varepsilon^2(\sigma-1)\Big)+C\varepsilon^3.
\end{align*}
Therefore, we obtain
\begin{align}\label{estimate grad}
v(x+h)+v(x-h)\leq 2v(x)+\varepsilon^2 a\sigma(\sigma -1)|x-z|^{\sigma-2} + C\varepsilon^3.
\end{align}
Also, for any vector $y\in B^{z-x}_\varepsilon$, we have $\langle y,x-z\rangle=0$. Hence, by a short calculation, we have
\begin{align*}
\begin{split}\label{estimate noise}
\vint_{B_\varepsilon^{x-z}}v(x+y)~d\mathcal L^{n-1}(y)&\leq v(x)+C\varepsilon^3+\frac{1}{2}a\sigma|x-z|^{\sigma-2}\vint_{B_\varepsilon^{x-z}}\abs{y}~d\mathcal L^{n-1}(y) \\
&=v(x)+C\varepsilon^3+\frac{1}{2}a\sigma|x-z|^{\sigma-2}\varepsilon^2\bigg(\frac{n-1}{n+1}\bigg).
\end{split}
\end{align*}
Thus, this inequality together with $B_\varepsilon^{z-x}=B_\varepsilon^{x-z}$ and \eqref{estimate grad} implies
\begin{align*}
&W\bigg(v;x,\varepsilon\frac{x-z}{|x-z|}\bigg)+W\bigg(v;x,-\varepsilon\frac{x-z}{|x-z|}\bigg) \\
&=\alpha(x)\bigg(v\bigg(x-\varepsilon\frac{x-z}{|x-z|}\bigg)+v\bigg(x+\varepsilon\frac{x-z}{|x-z|}\bigg)\bigg)\\
&\hspace{12pt} +2\beta(x)\vint_{B_\varepsilon^{z-x}}v(x+\tilde{h})~d\mathcal L^{n-1}(\tilde{h}) \\
&\leq 2 v(x)+\varepsilon^2 a\sigma|x-z|^{\sigma-2}\Big(\beta(x)+\alpha(x)(\sigma -1)\Big) +2 C\varepsilon^3. \qedhere
\end{align*}
\end{proof}
\end{lemma}

Next, we prove the main theorem of this section. To get the result, we need to assume some regularity on the boundary of the set $\Omega$.
\begin{brc}

There are universal constants $r_0,s\in (0,1)$ such that for all $r \in (0,r_0]$ and $y\in \partial \Omega$ there exists a ball
\[
\begin{split}
B_{s r}(z) \subset B_r(y) \setminus \Omega
\end{split}
\]
for some $z \in B_r(y) \setminus \Omega$.
\end{brc}

Assume that the set $\Omega$ satisfies this boundary condition. Then, the following theorem holds.
\begin{theorem}\label{boundary regularity}
Let $y\in \partial \Omega$ and $r\in (0,r_0]$ with $r_0\in (0,1)$ and the ball $B_{sr}(z)\subset  B_r(y)\setminus \Omega$ given by the boundary regularity condition. Let $u$ be the solution to \eqref{DPP} with continuous boundary data $F$. Then for all $\eta>0$, there exist $\varepsilon_0>0$ and $k\geq 1$ such that
\begin{equation}\label{boundary estimate}
 u(x_0)-\sup_{B_{4r}(z)\cap \Gamma_{\varepsilon,\varepsilon}} F<\eta
\end{equation}
for all $0<\varepsilon<\varepsilon_0<1$ and $x_0\in B_{4^{1-k}r}(y)\cap \ol{\Omega}_\varepsilon$.

\begin{proof}
The idea is to find a suitable barrier function so that by Lemma \ref{W lemma}, if $P_\I$ pulls towards the point $z\in B_r(y) \setminus \Omega$, the game process inside the barrier function is a supermartingale. Then by utilizing the properties of the barrier function, we get the result by iteration.

Choose a constant $0<\theta<1$, independent of $r$, such that
$$
\theta:=\frac{s^\sigma-2^\sigma}{s^\sigma-4^\sigma}
$$
with the parameter $s>0$ from the boundary condition and a parameter $\sigma<0$ that will be defined later. We extend the function $F$ continuously to the set $\Gamma_{1,1}$ and use the same notation for the extension. Then, we choose $k\geq 1$ big enough such that
$$
\theta^k\Big(\sup_{\Gamma_{1,1}} F-\inf_{\Gamma_{1,1}} F\Big)<\eta.
$$
In addition, we denote the constants
$$
b_U:=\sup_{\Gamma_{\varepsilon,\varepsilon}} F
$$
and
$$
b_{4r}:=\sup_{B_{4r}(z)\cap \Gamma_{\varepsilon,\varepsilon}} F.
$$
Thus for the chosen $k$, independent of $\varepsilon$, it holds
\begin{equation}\label{estimate k and theta}
\theta^k\big(b_U-b_{4r}\big)<\eta.
\end{equation}
We define a function $v_k$ such that
$$
v_k(x)=a|x-z|^\sigma+b
$$
in $ B_{4^{2-k} r}(z)\setminus \overline{B}_{4^{1-k}s r}(z)$. The constants $a\leq 0$ and $b\geq 0$ can be calculated from the boundary values
\begin{align}
\begin{split}\label{extended v}
v_k=\begin{cases}b_{4r}+\theta^{k-1}\big(b_U-b_{4r}\big) ~~&\text{on }\partial B_{4^{2-k} r}(z)\\
b_{4r}~~&\text{on }\partial B_{4^{1-k}s r}(z).
\end{cases}
\end{split}
\end{align}
If $b_U=b_{4r}$, it holds $a=0$ and $b=b_U$. Otherwise, the values are $a<0$ and $b\geq 0$. We consider the case with $a<0$, since the proof of the other case is clear.

We extend the function $v_k$ to the set $\mathbb R^n\setminus \overline{B}_{4^{1-k}s r-2\varepsilon}(z)$ and use the same notation for the extension. We may assume that $x_0 \in \Omega$, and observe that
\begin{equation}\label{eq:x0 distance z}
x_0\in B_{4^{1-k}r}(y)\cap \Omega \subset \Omega \cap B_{2\cdot4^{1-k}r}(z)\setminus \overline{B}_{4^{1-k}s r}(z).
\end{equation}
Assume that $P_{\textrm{II}}$ plays the game by pulling towards the point $z$ given a turn, i.e., he moves the game token by the vector $-\varepsilon(x_m-z)/|x_m-z|$, if he wins the $m$th toss. This strategy is denoted by $\mathcal{S}_\II^*$. Also, fix a strategy for $P_{\textrm{I}}$  and denote it by $\mathcal{S}_\I$.

By using Lemma \ref{W lemma} for all $m\geq1$, we can estimate
\begin{align*}
&\mathbb E_{\mathcal{S}_\I,\mathcal{S}_\II^*}^{x_0}\big[v_k(x_{m+1})|x_0,\dots,x_m\big] \\
&\leq \frac{1-\delta(x_m)}{2}\bigg(\sup_{|\nu|=\varepsilon} W(v_k;x_m,\nu)+W\bigg(v_k;x_m,-\varepsilon \frac{x_m-z}{|x_m-z|}\bigg)\bigg) + \delta(x_m)F(x_m) \\
& \leq  \frac{1-\delta(x_m)}{2}\bigg(2 v_k(x_m)+\varepsilon^2 a\sigma|x_m-z|^{\sigma-2}\cdot \\
&\hspace{64pt}\cdot\big(\beta(x_m)+\alpha(x_m)(\sigma -1)\big) + 2C\varepsilon^3 \bigg) 
 + \delta(x_m)F(x_m)
\end{align*}
for some $C>0$. Next, we need to choose the constant $\sigma<0$ small enough. Recall that $\alpha_{\text{min}}>0$. Let us fix the value
$$
\sigma:=\frac{2}{\alpha_{\text{min}}}(\alpha_{\text{min}}-1)
$$
implying
$$
\beta(x_m)+\alpha(x_m)(\sigma -1)<-1
$$
for all $x_m \in \Omega$. In addition, we have
\begin{align*}
&a\sigma|x_m-z|^{\sigma-2}>a\sigma\big(\diam(\Omega)+1\big)^{\sigma-2}>0
\end{align*}
for all $x_m \in \Omega$. Thus by choosing $\varepsilon_0:=\varepsilon_0(\alpha_{\text{min}},r,\Omega,k)>0$ small enough, we can ensure that
\begin{align*}
\mathbb E_{\mathcal{S}_\I,\mathcal{S}_\II^*}^{x_0}[v_k(x_{m+1})|x_0,\dots,x_m]&\leq \big(1-\delta(x_m)\big)v_k(x_m)+\delta(x_m)F(x_m)  \leq v_k(x_m)
\end{align*}
for all $\varepsilon<\varepsilon_0$. We have shown that the process
$$
M_m:=v_k(x_m)
$$
is a supermartingale, when $P_\II$ uses the strategy $\mathcal{S}_\II^*$ and $P_\I$ uses any strategy $\mathcal{S}_\I$.

Define a boundary function $F_{v_k}: \Gamma_{\varepsilon,\varepsilon} \to \mathbb R$ such that
$$
F_{v_k}=v_k|_{\Gamma_{\varepsilon,\varepsilon}}.
$$
By \Cref{value unique}, we have $u=u_\I$ with $u_\I$ the value function for $P_\I$ defined in \eqref{uIIIdef}. Since $F\leq F_{v_k}$, $(M_m)_{m=1}^\infty$ is a supermartingale, $F_{v_k}$ is bounded and $\tau<\infty$ almost surely, we can estimate with the help of the optimal stopping theorem
\begin{align*}
u(x_0)&=\sup_{\mathcal{S}_\I}\inf_{\mathcal{S}_\II} \mathbb E^{x_0}_{\mathcal{S}_\I,\mathcal{S}_\II} \big[F(x_\tau)\big] \leq \sup_{\mathcal{S}_\I} \mathbb E^{x_0}_{\mathcal{S}_\I,\mathcal{S}_\II^*} \big[F(x_\tau)\big] \\
&\leq \sup_{\mathcal{S}_\I} \mathbb E^{x_0}_{\mathcal{S}_\I,\mathcal{S}_\II^*} \big[F_{v_k}(x_\tau)\big] \leq v_k(x_0).
\end{align*}
By using the boundary values \eqref{extended v}, we can calculate the constants $a$ and $b$ in the function $v_k$ and deduce by \eqref{eq:x0 distance z} that
$$
v_k(x_0)\leq b_{4r}+\theta^k\big(b_U-b_{4r}\big).
$$
Hence by \eqref{estimate k and theta}, we have shown the estimate \eqref{boundary estimate}.
\end{proof}
\end{theorem}

\begin{corollary}
Let $\eta>0$ and let $u$ be the solution to \eqref{DPP} with continuous boundary data $F$. Then, there is a constant $\ol{r}\in(0,r_0]$ such that for all $r \in(0,\ol{r}]$ there exist constants $k\geq1$ and $\varepsilon_0>0$ such that for any $y\in \Gamma_{\varepsilon,\varepsilon}$ it holds
$$
u(x_0)-F(y)<\eta
$$
for all $\varepsilon<\varepsilon_0$ and $x_0 \in B_{4^{1-k}r}(y)\cap \ol{\Omega}_\varepsilon$.
\begin{proof}
First, assume that $y\in \partial \Om$. Theorem \ref{boundary regularity} implies that for any $r \in (0,r_0]$, there are constants $k\geq 1$ and $\varepsilon_0>0$ such that
$$
u(x_0)-\sup_{B_{4r}(z)\cap \Gamma_{\varepsilon,\varepsilon}} F<\frac{\eta}{10}
$$
for all $\varepsilon<\varepsilon_0$ and $x_0\in B_{4^{1-k}r}(y)\cap \ol{\Omega}_\varepsilon$. Let $y^*\in \ol{B}_{4r}(z)\cap \Gamma_{\varepsilon,\varepsilon}$ be such that
$$
\sup_{B_{4r}(z)\cap \Gamma_{\varepsilon,\varepsilon}} F<F(y^*)+\frac{\eta}{10}.
$$
The boundary function $F$ is continuous on the compact set $\Gamma_{\varepsilon,\varepsilon}$, so there is a modulus of continuity $\omega_F$ for the function $F$. Thus, we can estimate
\begin{align*}
u(x_0)-F(y)&=u_\varepsilon(x_0)-F(y^*)+F(y^*)-F(y) \\
&<u-\sup_{B_{4r}(z)\cap \Gamma_{\varepsilon,\varepsilon}} F+\frac{\eta}{10}+\omega_F\big(|y^*-y|\big) \\
&<\frac{\eta}{5}+\omega_F\big(|y^*-y|\big).
\end{align*}
It holds $|z-y|<r$ implying the estimate $|y^*-y|\leq |y^*-z|+|z-y|<5r$. We choose $\ol{r}>0$ so small that
$$
\omega_F(5r)<\frac{\eta}{10}
$$
for all $r<\ol{r}$. This yields that for any $r<\ol{r}$, we have
$$
u(x_0)-F(y)<\frac{\eta}{2}
$$
for all $\varepsilon<\varepsilon_0$ and $x_0 \in B_{4^{1-k}r}(y)\cap \ol{\Omega}_\varepsilon$.

Next, assume that $y\not\in \partial \Om$. Pick a point $y_b\in \partial \Om$ such that $y\in B_\varepsilon(y_b)$. We choose $\varepsilon_0>0$ so small that
$$
\omega_F(\varepsilon_0)<\frac{\eta}{2}.
$$
This implies that
$$
\big|F(y)-F(y_b)\big|\leq \omega_F\big(|y-y_b|\big)<\frac{\eta}{2}
$$
for all $\varepsilon<\varepsilon_0$. Since $y_b\in \partial \Om$, we can use the estimates above to get the result.
\end{proof}
\end{corollary}

\begin{remark}\label{the estimate from under}
By using a similar argument to \Cref{boundary regularity}, it holds for all $y\in \partial\Omega$ and $\eta>0$ that there exist $\varepsilon_0>0$ and $k\geq 1$ such that
\begin{equation*}
 u(x_0)-\inf_{B_{4r}(z)\cap \Gamma_{\varepsilon,\varepsilon}} F>-\eta
\end{equation*}
for all $0<\varepsilon<\varepsilon_0<1$ and $x_0\in B_{4^{1-k}r}(y)\cap \ol{\Omega}_\varepsilon$. Hence, we have for all $r>0$ small enough, $k\geq 1$ big enough and $\varepsilon>0$ small enough the lower bound
$$
u(x_0)-F(y)>-\eta
$$
for $y\in \Gamma_{\varepsilon,\varepsilon}$ and $x_0\in B_{4^{1-k}r}(y)\cap \ol{\Omega}_\varepsilon$.
\end{remark}

\section{Application}\label{sec:applications}
In this section, we prove that the uniform limit of functions satisfying \eqref{DPP} as $\varepsilon\to 0$ is a weak solution to the normalized homogeneous $p(x)$-Laplace equation
\begin{align}
\begin{split}\label{normplapeq}
\Delta_{p(x)}^Nu(x):=&~\Delta u(x)+\big(p(x)-2\big)\Delta^N_\infty u(x) \\
=&~\Delta u(x)-\Delta^N_\infty u(x)+\big(p(x)-1\big)\Delta^N_\infty u(x) \\
=&~0.
\end{split}
\end{align}
This equation is in a non divergence form so we define weak solutions via viscosity theory. There is a related version of the equation \eqref{normplapeq}, called a strong $p(x)$-Laplacian, in a divergence form, which has recently received attention and studied using distributional weak theory (see for example \cite{adamowiczh11,zhangz12,peres-llanosm13}). For some questions, the viscosity point of view is very natural in the sense that the equation \eqref{normplapeq} has the Pucci operator bounds used for example in Section 4 in \cite{caffarellic95}.

We define for all vectors $x,h\in \mathbb R^n$ and symmetric $n \times n$ matrices $X$
$$
\mathbb F_{p(x)}(x,h,X):=\tr(X)-\sum_{i,j}^n\frac{h_ih_j}{|h|^2}X_{ij}+\big(p(x)-1\big)\sum_{i,j}^n\frac{h_ih_j}{|h|^2}X_{ij}.
$$
These functions are discontinuous, when $h=0$. Therefore, we define viscosity solutions via semicontinuous extensions. For more details about the extensions, see for example \cite{giga06,evanss91}. We denote by $\lambda_{\text{min}}(X)$ and by $\lambda_{\text{max}}(X)$ the smallest and the largest eigenvalues of a symmetric matrix $X$.

\begin{definition}
A continuous function $u:\Omega \to \mathbb R$ is a viscosity solution to the equation \eqref{normplapeq}, if for all $x\in \Omega$ and $\phi\in C^2$ such that $u(x)=\phi(x)$ and $u(y)>\phi(y)$ for $y\not=x$ we have
\begin{align}\begin{split}\label{def viscosity} \left\{\!\!\!
\begin{array}{rclcl}
0 \!\!
&\geq& \!\!\mathbb F_{p(x)}(x,\nabla \phi(x),D^2\phi(x)),\!\!\!\!&\textrm{if}& \nabla
\phi(x)\neq 0, \\[2pt]
0\!\! &\geq&\!\! \lambda_{\text{min}}((p(x)-2)D^2\phi(x))+\tr(D^2\phi(x)),\!\!\!\!&\textrm{if}& \nabla \phi (x)=0.
\end{array} \right.
\end{split}
 \end{align}
We also require that for all $x\in \Omega$ and $\phi\in C^2$ such that $u(x)=\phi(x)$ and $u(y)<\phi(y)$ for $y\not=x$ all the inequalities are reversed, and we use $\lambda_{\text{max}}$ in the role of $\lambda_{\text{min}}$.
\end{definition}
It is equivalent to require that $u-\phi$ has a local strict minimum at $x$ instead of $u(x)=\phi(x)$ and $u(y)>\phi(y)$ for $y\not=x$ (see for example \cite{koike04}). Next, we prove that by passing to a subsequence if necessary, the value function of the game converges uniformly to a solution of the equation \eqref{normplapeq}. To prove that the limiting function $u$ is a viscosity solution to \eqref{normplapeq}, we use an argument similar to the stability principle for viscosity solutions. We apply the DPP \eqref{DPP} for a test function $\phi\in C^2$ and deduce the connection by utilizing the uniform convergence.

\begin{theorem}\label{theorem:converge}
Let $u_\varepsilon$ denote the unique continuous solution to \eqref{DPP} with $\varepsilon>0$ and with a continuous boundary function $F:\Gamma_{\varepsilon,\varepsilon}\to \mathbb R$. Then, there are a function $u:\ol{\Omega}\to \mathbb R$ and a subsequence $\{\varepsilon_i\}$ such that $u_{\varepsilon_i}$ converges uniformly to $u$ on $\ol{\Omega}$ and the function $u$ is a viscosity solution to \eqref{normplapeq} with the boundary data $F$.
\end{theorem}

\begin{proof}
To find the function $u$, we use a variant of the Arzel\`{a}-Ascoli's theorem (see for example \cite[p.\,15-16]{manfredipr12}). By Theorems \ref{LOCREG} and \ref{boundary regularity} together with \Cref{the estimate from under}, the assumptions for Arzel\`{a}-Ascoli's theorem are satisfied and hence, there exist a continuous function $u$ on $\ol{\Omega}$ with the boundary values $F$ and a subsequence $\{\varepsilon_i\}$ such that $u_{\varepsilon_i}\to u$ uniformly on $\ol{\Omega}$ as $i\to \infty$. Thus, it is enough to show that $u$ is a viscosity solution to \eqref{normplapeq}.

Let $x\in\Omega$ and $\phi\in C^2$ such that $u-\phi$ has a strict local minimum at $x$. Then, we have
$$
\inf_{B_r(x)}(u-\phi)=u(x)-\phi(x)<u(z)-\phi(z)
$$
for some $r>0$ and for all $z\in B_r(x)\setminus \{x\}$. The uniform convergence yields
$$
\inf_{B_r(x)}(u_\varepsilon-\phi)<u_\varepsilon(z)-\phi(z)
$$
for all $z\in B_r(x)\setminus \{x\}$ and for all $\varepsilon>0$ small enough. Thus, we can use the definition of the infimum and deduce that for all $\eta_\varepsilon>0$, there exists a point $x_\varepsilon\in B_r(x)\subset \Omega$ such that
$$
u_\varepsilon(x_\varepsilon)-\phi(x_\varepsilon)\leq u_\varepsilon(z)-\phi(z)+\eta_\varepsilon
$$
for all $z\in B_r(x)$ and $\varepsilon>0$ small enough with $x_\varepsilon \to x$ as $\varepsilon \to 0$. We define $\varphi:=\phi+u_\varepsilon(x_\varepsilon)-\phi(x_\varepsilon)$ so that
\begin{equation*}
\varphi(x_\varepsilon)=u_\varepsilon(x_\varepsilon)\text{ and }u_\varepsilon(z)\geq\varphi(z)-\eta_\varepsilon
\end{equation*}
for all $z\in B_r(x)$. Therefore, these together with the fact that $u_\varepsilon$ is a solution to \eqref{DPP} imply
\begin{align*}
u_\varepsilon(x_\varepsilon)=T_\varepsilon u_\varepsilon(x_\varepsilon)&\geq T_\varepsilon\varphi(x_\varepsilon)-(1-\delta(x_\varepsilon))\eta_\varepsilon \\
&=T_\varepsilon\phi(x_\varepsilon)-\eta_\varepsilon+u_\varepsilon(x_\varepsilon)-\phi(x_\varepsilon)+\delta(x_\varepsilon)\Lambda_\varepsilon,
\end{align*}
where we use the monotonicity of $T_\varepsilon$ and denote $\Lambda_\varepsilon:=\eta_\varepsilon+\phi(x_\varepsilon)-u_\varepsilon(x_\varepsilon)$. This inequality yields
\begin{equation}\label{varphi est}
\eta_\varepsilon \geq T_\varepsilon\phi(x_\varepsilon)-\phi(x_\varepsilon)+\delta(x_\varepsilon)\Lambda_\varepsilon.
\end{equation}

By the Taylor's expansion of $\phi$ at $x_\varepsilon$ with $\abs\nu=1$, we get
\begin{align}
				\frac{1}{2}\phi(x_\varepsilon+\varepsilon\nu)+\frac{1}{2}\phi(x_\varepsilon-\varepsilon\nu)=\phi(x_\varepsilon)+\frac{\varepsilon^2}{2}\prodin{D^2\phi(x_\varepsilon) \nu}{\nu}+o(\varepsilon^2), \label{Taylor1}\\
				\dashint_{B_\varepsilon^\nu} \phi(x_\varepsilon+h)\dL(h)=\phi(x_\varepsilon)+\frac{\varepsilon^2}{2(n+1)}\Delta_{\nu^\bot}\phi(x_\varepsilon)+o(\varepsilon^2).\label{Taylor2}
\end{align}
In \eqref{Taylor2}, we utilize the orthonormal basis $\mathcal V$ including $\nu$ and an orthonormal basis for $\nu^\bot$ to obtain
$$
\frac{1}{2}\dashint_{B_\varepsilon^\nu} \prodin{D^2\phi(x_\varepsilon)h}{h}\dL(h)=\frac{\varepsilon^2}{2(n+1)}\Delta_{\nu^\bot}\phi(x_\varepsilon)
$$
in a similar way to \cite{manfredipr10}. Here, the operator $\Delta_{\nu^\bot}$ denotes the Laplacian on the plane $\nu^\bot$, i.e.,
$$
\Delta_{\nu^\bot}\phi(x_\varepsilon)=\sum_{j=2}^{n}\prodin{D^2\phi(x_\varepsilon) \nu_j}{\nu_j}
$$
with $\nu_2,\dots,\nu_n$ the orthonormal basis vectors for $\nu^\bot$. Observe that
\begin{equation}\label{laplacian coordinates}
\Delta \phi(z)=\tr(D^2\phi(z))=\Delta_{\nu^\bot}\phi(z)+\prodin{D^2\phi(z) \nu}{\nu}
\end{equation}
for any $|\nu|=1$ and $z\in \Omega$. To see this, we apply the orthonormal basis $\mathcal V$ and a change of variables $x_1=\nu,x_2=\nu_2,\dots, x_n=\nu_n$. Then, we deduce the equation \eqref{laplacian coordinates} by the chain rule.

There exists a vector $\nu_{\min}:=\nu_{\min}(\varepsilon)$ minimizing
\begin{equation}\label{Wphi}
				\alpha(x_\varepsilon)\phi(x_\varepsilon+\varepsilon\nu)+\beta(x_\varepsilon)\dashint_{B_\varepsilon^\nu} \phi(x_\varepsilon+h)\dL(h)
\end{equation}
with $\abs\nu=1$. Thus by \eqref{Taylor1} and \eqref{Taylor2} with $\nu=\nu_{\min}$ and the fact that $-\nu^\bot\equiv\nu^\bot$, we obtain
\begin{align*}
				T_\varepsilon\phi(x_\varepsilon) =&\frac{1-\delta(x_\varepsilon)}{2}\sup_{\abs\nu=1}\pare{\alpha(x_\varepsilon)\phi(x_\varepsilon+\varepsilon\nu)+\beta(x_\varepsilon)\dashint_{B_\varepsilon^\nu} \phi(x_\varepsilon+h)\dL(h)} \\
				& + \frac{1-\delta(x_\varepsilon)}{2}\inf_{\abs\nu=1}\pare{\alpha(x_\varepsilon)\phi(x_\varepsilon+\varepsilon\nu)+\beta(x_\varepsilon)\dashint_{B_\varepsilon^\nu}\phi(x_\varepsilon+h)\dL(h)} \\
				& +  \delta(x_\varepsilon) F(x_\varepsilon) \\
							~\geq&  (1-\delta(x_\varepsilon))\frac{\alpha(x_\varepsilon)}{2}\braces{\phi(x_\varepsilon+\varepsilon\nu_{\min})+\phi(x_\varepsilon-\varepsilon \nu_{\min})} \\
				~&+(1-\delta(x_\varepsilon))\beta(x_\varepsilon)\dashint_{B_\varepsilon^{\nu_{\min}}} \phi(x_\varepsilon+h)\dL(h) +\delta(x_\varepsilon)F(x_\varepsilon) +o(\varepsilon^2) \\
								~= & (1-\delta(x_\varepsilon))\phi(x_\varepsilon) +(1-\delta(x_\varepsilon))\frac{\beta(x_\varepsilon)\varepsilon^2}{2(n+1)}\Big\{\Delta_{\nu_{\min}^\bot}\phi(x_\varepsilon)\\
				& \hspace{110pt} +(p(x_\varepsilon)-1)\prodin{D^2\phi(x_\varepsilon) \nu_{\min}}{\nu_{\min}}\Big\} \\
				~& + \delta(x_\varepsilon)F(x_\varepsilon) +o(\varepsilon^2).
\end{align*}
By this estimate and \eqref{varphi est}, we have
\begin{equation}
\begin{split} \label{DPPphi}
				 \eta_\varepsilon \geq ~& -\phi(x_\varepsilon)+(1-\delta(x_\varepsilon))\Bigg\{\phi(x_\varepsilon)+\frac{\beta(x_\varepsilon)\varepsilon^2}{2(n+1)}\Big[\Delta_{\nu_{\min}^\bot}\phi(x_\varepsilon) \\
				&\hspace{143pt} +(p(x_\varepsilon)-1)\prodin{D^2\phi(x_\varepsilon) \nu_{\min}}{\nu_{\min}}\Big]\Bigg\} \\
				~& + \delta(x_\varepsilon)\big(F(x_\varepsilon)+\Lambda_\varepsilon\big) +o(\varepsilon^2).
\end{split}
\end{equation}

First, assume that $|\nabla\phi(x)|\neq 0$. Then by $x_\varepsilon\to x$ as $\varepsilon \to 0$, it turns out that
\begin{equation}\label{v0}
				\nu_{\min}\rightarrow -\frac{\nabla\phi(x)}{|\nabla\phi(x)|}=:\nu_{\min}^*
\end{equation}
as $\varepsilon \to 0$. Here, we need the fact that $\alpha_{\text{min}}>0$, i.e., $p_{\text{min}}>1$. In addition by \eqref{laplacian coordinates}, we have
\begin{align}\label{deltapiv0}
				\Delta_{(\nu_{\min}^*)^\bot}\phi(x)+(p(x)-1)\prodin{D^2\phi(x) \nu_{\min}^*}{\nu_{\min}^*}=\Delta_{p(x)}^N\phi(x).
\end{align}
Choose $\eta_\varepsilon=o(\varepsilon^2)$, divide both sides in \eqref{DPPphi} by $\varepsilon^2$ and let $\varepsilon\rightarrow 0$. Therefore by \eqref{DPPphi}, \eqref{v0}, \eqref{deltapiv0} and the facts that $x_\varepsilon\to x $ and $\delta(x_\varepsilon)\varepsilon^{-2}\rightarrow 0$ as $\varepsilon \to 0$, we obtain
\begin{equation*}
				0 \geq \frac{\beta(x)}{2(n+1)}\Delta_{p(x)}^N\phi(x).
\end{equation*}
By $\beta(x)\geq \beta_{\text{min}}>0$, we get the inequality $\Delta_{p(x)}^N\phi(x)\leq 0$.

Next, assume that $|\nabla\phi(x)|= 0$. By \eqref{laplacian coordinates}, we have
\begin{align}
\begin{split}\label{last one}
&\Delta_{\nu_{\min}^\bot}\phi(x_\varepsilon) +(p(x_\varepsilon)-1)\prodin{D^2\phi(x_\varepsilon) \nu_{\min}}{\nu_{\min}}\\
&=\Delta \phi(x)+(p(x_\varepsilon)-2)\prodin{D^2\phi(x_\varepsilon) \nu_{\min}}{\nu_{\min}}.
\end{split}
\end{align}
Assume that $p(x)>2$. Then by the continuity of $p$ and $x_\varepsilon \to x$ as $\varepsilon \to 0$, we have $p(x_\varepsilon)>2$ for all $\varepsilon$ small enough. Thus, we can estimate
$$
(p(x_\varepsilon)-2)\prodin{D^2\phi(x_\varepsilon) \nu_{\min}}{\nu_{\min}}\geq (p(x_\varepsilon)-2)\lambda_{\text{min}}(D^2\phi(x_\varepsilon))
$$
for all $\varepsilon$ small enough. As above, this estimate together with \eqref{DPPphi}, \eqref{last one} and the continuity of $z \mapsto \lambda_{\text{min}}(D^2\phi(z))$ imply
$$
0 \geq \Delta \phi(x)+\lambda_{\text{min}}((p(x)-2)D^2\phi(x)).
$$
The cases $p(x)=2$ and $p(x)<2$ can be treated in a similar fashion.

To show the required reverse inequality, we choose $\abs{\nu_{\max}}=1$ such that it maximizes \eqref{Wphi}, and we consider the reverse inequality with $\phi\in C^2$ such that $u-\phi$ has a strict maximum at $x$. The proof is analogous to the above.
\end{proof}

\section*{Acknowledgements.} Part of this research was done during a visit of \'{A}.\ A.\ to the University of Jyv\"{a}skyl\"{a} in 2015. \'{A}.\ A.\ was partially supported by grants MTM2011-24606, MTM2014-51824-P and 2014 SGR 75. M. P. was supported by the Academy of Finland project \#260791.

\medskip

\end{document}